\documentclass[a4paper,11pt]{article}
\usepackage{amsmath,amssymb}
\usepackage{epsfig}
\usepackage{algorithm,algorithmic}
\usepackage{amsthm}
\usepackage{geometry}
\usepackage{lineno}

\makeatletter
\@addtoreset{equation}{section}

\makeatother

\numberwithin{equation}{section}
\newtheorem{thm}{\bfseries Theorem}[section]
\newtheorem{lem}[thm]{\bfseries Lemma}       
\newtheorem{prop}[thm]{\bfseries Proposition} 
\newtheorem{cor}[thm]{\bfseries Corollary}     
\theoremstyle{definition}
\newtheorem{algo}[thm]{\bfseries Algorithm}
\newtheorem{ex}[thm]{\bfseries Example}
\newtheorem{rem}[thm]{\bfseries Remark}

\title{A representation of antimatroids by Horn rules and its application to educational systems} 
\author{Hiyori Yoshikawa\footnote{Department of Mathematical Informatics, 
		Graduate School of Information Science and Technology,   
		University of Tokyo, Tokyo, 113-8656, Japan.
		\texttt{hyr480.yskw1103@gmail.com}}
	\quad 
Hiroshi Hirai\footnote{
	Department of Mathematical Informatics, 
	Graduate School of Information Science and Technology,   
	University of Tokyo, Tokyo, 113-8656, Japan.
	\texttt{hirai@mist.i.u-tokyo.ac.jp}}
 \quad
Kazuhisa Makino\footnote{
Research Institute for Mathematical Science,
Kyoto University,
Kyoto 606--8502, Japan
\texttt{makino@kurims.kyoto-u.ac.jp}}}

\begin{document}
	
\maketitle

\begin{abstract}
We study a representation of an antimatroid by Horn rules, 
motivated by its recent application to computer-aided educational systems. 
We associate any set $\mathcal{R}$ of Horn rules with 
the unique maximal antimatroid $\mathcal{A}(\mathcal{R})$ that is contained in
the union-closed family $\mathcal{K}(\mathcal{R})$ naturally determined by ${\cal R}$.
We address algorithmic and Boolean function theoretic aspects 
on the association ${\cal R} \mapsto \mathcal{A}(\mathcal{R})$, 
where ${\cal R}$ is viewed as the input.
We present linear time algorithms to solve the membership problem 
and the inference problem for ${\cal A}({\cal R})$. 
We also provide efficient algorithms for generating
all members and all implicates of ${\cal A}({\cal R})$.
We show that this representation is essentially equivalent to 
the Korte-Lov\'{a}sz representation of antimatroids by rooted sets.
Based on the equivalence,  
we provide a quadratic time algorithm to construct 
the uniquely-determined minimal representation.
These results have potential applications to computer-aided educational systems, 
where an antimatroid is used as a model of the space of possible knowledge states of learners, 
and is constructed by giving Horn queries to a human expert. 
\end{abstract}

\begin{quote}
{\bf Keywords:} Antimatroids; Horn rules; Implicational systems; Learning spaces; Knowledge spaces; Educational systems
\end{quote}

\section{Introduction}

An {\em antimatroid} is a family $\mathcal{K}$ of subsets of a finite set $Q$ satisfying the following conditions: 
\begin{description}
\item[(Union-closedness)] For members $X,Y$ of $\mathcal{K}$, the union $X\cup Y$ is also a member of $\mathcal{K}$.
\item[(Accessibility)] For every nonempty member $X$ of $\mathcal{K}$, there exists an element $x$ in $X$ such that $X\setminus \{x\}$ is a member of $\mathcal{K}$.
\end{description}
(We do not impose the usual condition $Q \in {\cal K}$.)
An antimatroid (or its dual, {\em convex geometry}) is an axiomatic abstraction of a finite point set in Euclidean space, and 
ubiquitously arises from various areas of discrete mathematics and theoretical computer science.
Examples appear from graph search, tree shelling, posets, and so on~\cite{Dietrich89,EdelmanJamison85}. 
One of the major applications of antimatroids is 
the analysis of greedily solvable structures in combinatorial optimization; 
see~\cite{KLS1991}.

A remarkable application of antimatroids has been emerging 
from the design of computer-aided education systems~\cite{DF1999,FADEX13, FD2011}.
In {\em Knowledge Space Theory (KST)}, 
an antimatroid is called a {\em learning space}, whereas a union-closed family is called a {\em knowledge space}.
They are used as mathematical models of the space of all possible knowledge states of learners, 
where the ground set $Q$ is a set of questions and 
the knowledge state of a learner is associated with 
the subset of questions that he/she answers correctly.
KST-based educational systems have already been in practical use, e.g., ALEKS\footnote{http://www.aleks.com/.}.  
In this application, the size of antimatroids can be quite large. 
Thus efficient representation, construction, and implementation of antimatroids are of great importance.

In the literature of KST, 
Dowling-M\"{u}ller~\cite{Dowling93, M1989} and Koppen and Doignon~\cite{KD1990} 
independently introduced a rule-based representation of union-closed families 
by certain binary relations, called {\em entailments}.
They established a Galois connection between union-closed families and entailments. 
In fact, their result may be viewed as a sharpening of 
the representation of families of 
Horn rules (or Horn formulas), 
a well-known concept in artificial intelligence and 
Boolean function theory; see \cite[Chapter 6]{Booleanfunctions}.
By a {\em Horn rule} (or {\em rule})
we mean a pair $(A,q)$ of a set $A \subseteq Q$ and an element $q$ in $Q$.
We say that {\em a rule $(A,q)$ accepts 
a subset $X$} if $q\in X$ implies $X\cap A\neq\emptyset$.  
For a set $\mathcal{R}$ of rules, let $\mathcal{K}(\mathcal{R})$ denote the set of subsets
accepted by all rules in $\mathcal{R}$. 
A classical result~\cite{Horn51,M1943} in formal logic says that 
a family ${\cal K}$ of subsets (including $\emptyset$) 
is union-closed if and only if it is represented 
by a set ${\cal R}$ of Horn rules as ${\cal K} = {\cal K}({\cal R})$; see~\cite[Theorem 6.6]{Booleanfunctions}.
 {\em Theory of implicational systems}~\cite{BM10,CM03,Wild94,Wild14survey} provides 
 a unified and systematic approach to this classical result, generalizations, and ramifications, 
 obtained in different fields of mathematical sciences 
(e.g., formal logic, lattice theory, formal concept analysis, relational database, and KST).
In this theory,  a Horn rule is called a {\em unit implication}, 
and a set ${\cal R}$ of Horn rules with ${\cal K} = {\cal K}({\cal R})$
is called a {\em unit implicational basis} of~${\cal K}$.

By definition, an antimatroid is a union-closed family.
Thus an antimatroid ${\cal A}$ can be realized in computers 
by maintaining some implicational base ${\cal R}$ of ${\cal A} = {\cal K}({\cal R})$.
Various operations on ${\cal K}({\cal R})$ 
can be efficiently conducted by accessing ${\cal R}$.
In practice, 
a large family can often be represented by a small set of rules.
Also in the literature of implicational systems,
several  ``useful" or ``compact" implicational bases 
of antimatroids (or convex geometries) and related closure systems
have been investigated; see \cite{Adaricheva12, AN14, ANR13, KN13, Wild94}.
However this way of representing an antimatroid by ${\cal K}({\cal R})$ has one obvious drawback:
Not every set of rules corresponds to an antimatroid. 
Therefore we need a special care to keep ${\cal K}({\cal R})$ 
an antimatroid when ${\cal R}$ is frequently varied by addition/deletion. 
Such a situation naturally occurs in the design of KST-based educational systems, 
and this drawback has been one of main difficulties for practical use of antimatroids.

In this paper, we overcome this drawback by another way of associating 
any set of rules with an antimatroid.
Our starting point is the following:
\begin{quote}
	{\em There exists a unique maximal antimatroid contained in any union-closed family.}
\end{quote}
This fundamental fact was recently noticed by Doignon~\cite[p. 14]{Doignon14ICFCA} in KST, 
and is a direct corollary of a classical result~\cite[Theorem 2.2]{Edelman80} of
Edelman that for two antimatroids ${\cal A}, {\cal A}'$
in the same ground set, the family 
$\{ X \cup X' \mid X \in {\cal A}, X' \in {\cal A}'\}$ is again an antimatroid.
We will use the following explicit characterization of this maximal antimatroid.
For finite sets $X$ and $Y$ with $Y\subseteq X$, a {\em tight path from $Y$ to $X$} is a sequence $Y = Y_{0}, Y_{1},\ldots , Y_{k} =X$ of subsets of $Q$ satisfying $Y_{i}\subseteq Y_{i+1}$ and $|Y_{i+1}\setminus Y_{i}|=1$ for $i=0, \ldots , k-1$. For a union-closed family $\mathcal{K}$ on $Q$, let $\check{\mathcal{K}}$ 
denote the family of subsets $K$ in $\mathcal{K}$ 
such that there is a tight path from $\emptyset$ to $K$ in $\mathcal{K}$. 
Then it holds:
%
\begin{thm}\label{thm:uniqueness}
For a union-closed family $\mathcal{K}$, the family $\check{\mathcal{K}}$ 
is the unique maximal antimatroid contained in $\mathcal{K}$. 
\end{thm}
For a set ${\cal R}$ of rules, we define ${\cal A}({\cal R})$ 
as the maximal antimatroid $\check{\cal K}({\cal R})$
in the union-closed family ${\cal K}({\cal R})$. 

The main subjects of this paper are fundamental algorithmic aspects on
the association ${\cal R} \mapsto {\cal A}({\cal R})$,
and their implications to the educational system design.
To begin with,
let us formalize our algorithmic setting.
We are given a set ${\cal R}$ of rules as an input, 
where the size of ${\cal R}$ 
is its coding length  $l({\cal R}) := \sum_{(A,q)\in\mathcal{R}}(|A|+1)$.
We address algorithms and computational complexity 
for basic problems of handling ${\cal A}({\cal R})$ by ${\cal R}$.
Notice that this is different from a standard setting in implicational systems: 
an implication basis of a family in question is given.
Indeed, ${\cal R}$ is not necessarily 
an implicational basis of ${\cal A}({\cal R})$.

We first consider the membership problem for ${\cal A}({\cal R})$:
\begin{description}
\item[Membership problem]\ \\ 
\begin{description}
	\vspace{-3ex}
\item[{\rm Input:}]  A set ${\cal R}$ of rules and a set $X$. 
\item[{\rm Task:}] Determine whether $X$ belongs to $\mathcal{A}(\mathcal{R})$.
\end{description}
\end{description}
Whereas the membership problem for ${\cal K}({\cal R})$ 
is easily solved (in linear time), 
the computational complexity of the membership problem for ${\cal A}({\cal R})$ 
is not trivial,  
and is in NP since a tight path 
from $\emptyset$ to $X$ in $\mathcal{K}(\mathcal{R})$ is a polynomial certificate. 
If the membership problem can be solved efficiently, 
then one might say that an antimatroid ${\cal A}$ can be realized in computers 
by maintaining ${\cal R}$ with ${\cal A} = {\cal A}({\cal R})$.
We show that the membership problem for ${\cal A}({\cal R})$ can also be solved in linear time.
\begin{thm}\label{thm:membership}
	The membership problem for ${\cal A}({\cal R})$ can be solved in linear time. 
\end{thm}
Based on this linear time membership algorithm, 
we will provide an efficient algorithm to enumerate all members of ${\cal A}({\cal R})$.

We next consider the inference problem, which is motivated 
by the query learning for educational systems; see below. 
A rule $(A, q)$ is called an {\em implicate} of a family $\mathcal{K}$ 
if $(A,q)$ accepts all the members of $\mathcal{K}$. 
\begin{description}
\item[Inference problem] \ \\ 
\begin{description}
\vspace{-3ex}
\item[\rm Input: ] A set $\mathcal{R}$ of rules and a rule $(A,q)$.
\item[\rm Task: ] Determine whether $(A,q)$ is an implicate of $\mathcal{A}(\mathcal{R})$.
\end{description}
\end{description}
We show that this problem can also be solved efficiently.
\begin{thm}\label{thm:inference}
The inference problem for ${\cal A}({\cal R})$ can be solved in linear time. 
\end{thm}

It turns out that this construction ${\cal R} \mapsto {\cal A}({\cal R})$ of an antimatroid 
is essentially equivalent to the construction of an antimatroid 
from rooted sets or circuits by Korte and Lov\'{a}sz \cite{KL1984}; 
see \cite[Section III. 3]{KLS1991}. 
We will establish this equivalence. 
Korte and Lov\'{a}sz showed the existence of the unique minimal representation of 
antimatroids by special circuits, called {\em critical circuits}. 
Translating their result, 
for an antimatroid ${\cal A}$,
there is a uniquel minimal set ${\cal R^*}$ of rules 
such that ${\cal A}({\cal R^*}) = {\cal A}$, where 
${\cal R^*}$ is minimal in its cardinality as well as its size. 
A rule in ${\cal R^*}$ is said to be {\em critical} for ${\cal A}$.
We will provide a quadratic time algorithm to construct 
this minimal set ${\cal R^*}$ from a given ${\cal R}$.
\begin{thm}\label{thm:critical}
For a given set $\mathcal{R}$ of rules, 
the set of critical rules for ${\cal A}({\cal R})$ can be obtained in quadratic time.
\end{thm}
As an application, 
we can determine, in quadratic time, whether two sets of rules 
define the same antimatroid. 

The representation ${\cal R} \mapsto {\cal A}({\cal R})$ 
fits naturally into the query learning of antimatroids 
arising from the design of KST-based educational systems, 
which is our practical motivation of this paper.
Koppen~\cite{K1993} and others~\cite{KD1990,M1989} considered a procedure {\em QUERY} 
to build a space of knowledge states
by asking a series of queries $(A_1,q_1), (A_2,q_2),\ldots,$ to a human expert. 
Here a query $(A, q)$ is the question:  
{\em does a learner fail to solve question $q$ provided he/she fails to solve every question in $A$?}
Then, for the set ${\cal R}$ of `yes' queries,  
the space of knowledge states is estimated as~${\cal K}({\cal R})$.
The QUERY procedure is designed for the case where the space is modeled as 
a union-closed family (a knowledge space).
Therefore, the resulting ${\cal K}({\cal R})$ is not necessarily an antimatroid. 
In practical situations, however, educational systems 
need to use an antimatroid (a learning space) as a model.
This leads to the following question given by Doignon and Falmagne in \cite[p. 335]{FD2011}: 
\begin{quote}
	{\em This raises the following problem: assuming that,
		except for errors, the responses to the queries are dictated by a latent learning
		space ${\cal L}$, can a learning space approximating ${\cal L}$ be derived by the querying
		method through some elaboration of QUERY~?}
\end{quote}
They developed a relatively complicated adaptation of QUERY, called {\em adapted QUERY}, 
that always keeps the estimated space an antimatroid  
by careful managing of `yes' queries and the surmise function; see \cite[Chapter 16]{FD2011} for detail.
 
Our results suggest a simple revision of QUERY  
{\em to use ${\cal A}({\cal R})$ instead of ${\cal K}({\cal R})$}.
Actually this approach was also suggested by Doignon under the name of {\em adjusted QUERY}, 
though an effective way of handling~${\cal A}({\cal R})$ 
was ``{\em still under investigation}"~\cite[p.14]{Doignon14ICFCA}.
Now the adjusted QUERY is efficiently implementable.
We believe that this is a desired elaboration of QUERY, which affirmatively answers the above question.
Indeed, by Theorem~\ref{thm:uniqueness}, 
the resulting space ${\cal A}({\cal R})$ 
is always an antimatroid that includes the target antimatroid ${\cal L}$,
 and might be a reasonable approximation of ${\cal L}$.
As seen in Theorem~\ref{thm:membership},  
the association ${\cal R} \mapsto {\cal A}({\cal R})$ is manageable in computer.
Moreover we can avoid giving redundant queries to the expert 
by making use of the algorithm in Theorem~\ref{thm:inference}.

\paragraph{Related work.} 
Eppstein, Falmagne, and Uzun~\cite{EFU09} consider a different approach to 
the above question of Doignon and Falmagne according to base families of antimatroids; see also \cite[Section 16.3]{FD2011}.
Here the {\em base} of a (union-closed) family ${\cal K}$ consists of members of ${\cal K}$
that are not able to be represented as a union of other members.
Clearly the base ${\cal B}$ of ${\cal K}$ can recover the original ${\cal K}$ completely.
Eppstein, Falmagne, and Uzun study 
several algorithmic questions on the base of an antimatroid 
(a well-graded family, more generally).
They developed a polynomial time algorithm to determine whether given a family ${\cal B}$
is the base of an antimatroid.
Moreover they also provided a polynomial time algorithm to construct (the base of) 
a minimal antimatroid containing given a (union-closed) family.
These results lead to another elaboration of the QUERY algorithm, which
first estimates a union-closed family ${\cal K}$ 
by the original QUERY algorithm, and, from the base of ${\cal K}$, 
constructs and outputs a minimal antimatroid ${\cal A}$ containing ${\cal K}$. 
Our approach may be viewed as a counter part of theirs, since it constructs
a maximal antimatroid contained in a union-closed family.

Wild~\cite{Wild14} applies a compression technique to the antimatroid construction in KST.
This technique encodes a union-closed family ${\cal L}$ 
into a $\{0,1,2,n\}$-valued matrix
for which each row vector is a compressed expression of a subfamily of ${\cal L}$ 
and ${\cal L}$ is the disjoint union of these subfamilies.
This matrix is constructed from an implicational basis of ${\cal L}$.
He discusses how to work the adapted QUERY with this encoding.
It would be an interesting research direction to 
incorporate this compression technique into our revision (or adjusted QUERY).

\paragraph{Organization.}
The rest of the paper is organized as follows. 
In Section~\ref{sec:preliminaries}, 
we give preliminary arguments including 
a proof of Theorem~\ref{thm:uniqueness}.
We also summarize some basic relationship among Horn rules, entailments, 
closure operators, and convex geometries, with the help of results in implicational systems,  
and then explain the Korte-Lov\'asz representation.
In Section~\ref{sec:algorithms}, we give algorithmic results.
We present algorithms for Theorems~\ref{thm:membership}, \ref{thm:inference}, and \ref{thm:critical}.
Also we give an efficient algorithm to generate all members of ${\cal A}({\cal R})$ 
and a resolution-type algorithm to generate all implicates of ${\cal A}({\cal R})$.
In Section~\ref{sec:application}, 
we discuss the above revised QUERY algorithm in more detail. 
Preliminary experimental results show that 
the revised QUERY algorithm reduces $30\%$ of queries that are needed 
by the original QUERY when the target space is an antimatroid.
In Section~\ref{sec:concluding}, 
we mention open problems and future research issues.

\section{Preliminaries}\label{sec:preliminaries}
Throughout the paper, $Q$ denotes a finite set.
By a family we mean a family of subsets of $Q$.  
The union $A \cup \{a\}$ of a set $A$ and an element $a$ is denoted by $A + a$.
The difference $A \setminus \{a\}$ is denoted by $A - a$. 
We assume that any union-closed family ${\cal K}$ considered in this paper contains the empty set, i.e.,
\[
\emptyset \in {\cal K},
\]
whereas any intersection-closed family contains the whole set $Q$.
For a union-closed family ${\cal K}$ 
and a set $X \subseteq Q$, there uniquely 
exists a maximal subset $Y \in {\cal K}$ contained in $X$.
This $Y$ is denoted by $X^{\circ}$.

An antimatroid is usually defined as a family ${\cal K}$ on $Q$ 
satisfying ({\bf Union-closedness}), ({\bf Accessibility}), and
\[
Q \in {\cal K}
\]
We here call such an antimatroid {\em proper}.
For an improper antimatroid ${\cal A}$, 
every subset is contained in $Q^{\circ} (\in {\cal A})$.
Hence ${\cal A}$ is regarded as a family on $Q^{\circ}$, and is a proper antimatroid on $Q^{\circ}$.
Thus known results and properties for (proper) antimatroids 
are easily adjusted for improper ones.
We remark that a union-closed family may or may not contain a proper antimatroid.
Notice that a union-closed family ${\cal K}$ contains a proper antimatroid if 
and only if there is a tight path from $\emptyset$ to $Q$ in ${\cal K}$.

As explained in several papers of implicational systems,
a set ${\cal R}$ of Horn rules is naturally identified with a (pure) Horn Boolean CNF $\varphi$, 
and ${\cal K}({\cal R})$ is identified with the set of true points of $\varphi$; 
see \cite[Section 5]{BM10} and \cite[Section 3.4]{Wild14survey}.
Various Boolean function theoretic concepts and algorithms are easily adapted to our setting.
In Appendix, we briefly summarize this relation to Horn Boolean CNFs.

\subsection{The unique maximal antimatroid in a union-closed family}
Here we prove Theorem \ref{thm:uniqueness}.
Let ${\cal K}$ be a union-closed family.
We first show that $\check{\mathcal{K}}$ is an antimatroid. 
Since the accessibility immediately follows from the definition of $\check{\cal K}$, 
it suffices to show that $\check{\mathcal{K}}$ is union-closed. 
Let $X$ and $Y$ be members of $\check{\mathcal{K}}$. 
By definition, there are tight paths $\emptyset = X_{0}, X_{1}, \ldots , X_{k}=X$ and $\emptyset =Y_{0}, Y_{1}, \ldots , Y_{m}=Y$ in $\mathcal{K}$. 
Then the distinct members of $\{X\cup Y_{i}\mid i=0,\ldots, m\}$ form a tight path from $X$ to $X\cup Y$. 
By the union-closedness of $\mathcal{K}$, all of them are members of $\mathcal{K}$. 
Combining it with the tight path from $\emptyset$ to $X$, we obtain a tight path from $\emptyset$ to $X\cup Y$ in $\mathcal{K}$. 
Hence we have $X\cup Y \in \check{\mathcal{K}}$. 

Finally we show the maximality of $\check{\mathcal{K}}$. Let $\mathcal{L}$ be an antimatroid contained by $\mathcal{K}$. 
By the accessibility, for every member $X$ of $\mathcal{L}$, there is a tight path from $\emptyset$ to $X$ in $\mathcal{K}$. 
Hence $\mathcal{L}\subseteq \check{\mathcal{K}}$, and 
Theorem~\ref{thm:uniqueness} is proved.

\begin{ex}
	Let $Q=\{0,1,2,3,4,5,6\}$.  Let ${\cal R}$ be the set of rules consisting of
	\begin{eqnarray*}
&&		(\{0, 1, 2, 3 \},6), (\{0, 1, 3, 5, 6 \},2), (\{0, 3, 4, 5, 6 \},2), (\{0, 3, 5, 6 \},1),  (\{0, 5, 6 \},4), (\{6\},0), \\ 
&& (\{1, 2, 3 \},0), (\{1, 4, 5 \},2), (\{1, 5 \},4), (\{2, 3, 4 \},1), (\{2, 3, 6 \},4), (\{2, 5, 6 \},4), \mbox{and } (\{2, 6 \},3).
\end{eqnarray*}	
As in Figure \ref{fig:KRAR}, there are 26 members in ${\cal A}({\cal R})$ and 53 members in ${\cal K}({\cal R})$,
	where members in ${\cal A}({\cal R})$ are colored gray. 
\end{ex}

\begin{figure}
\centering
\includegraphics[scale=0.2]{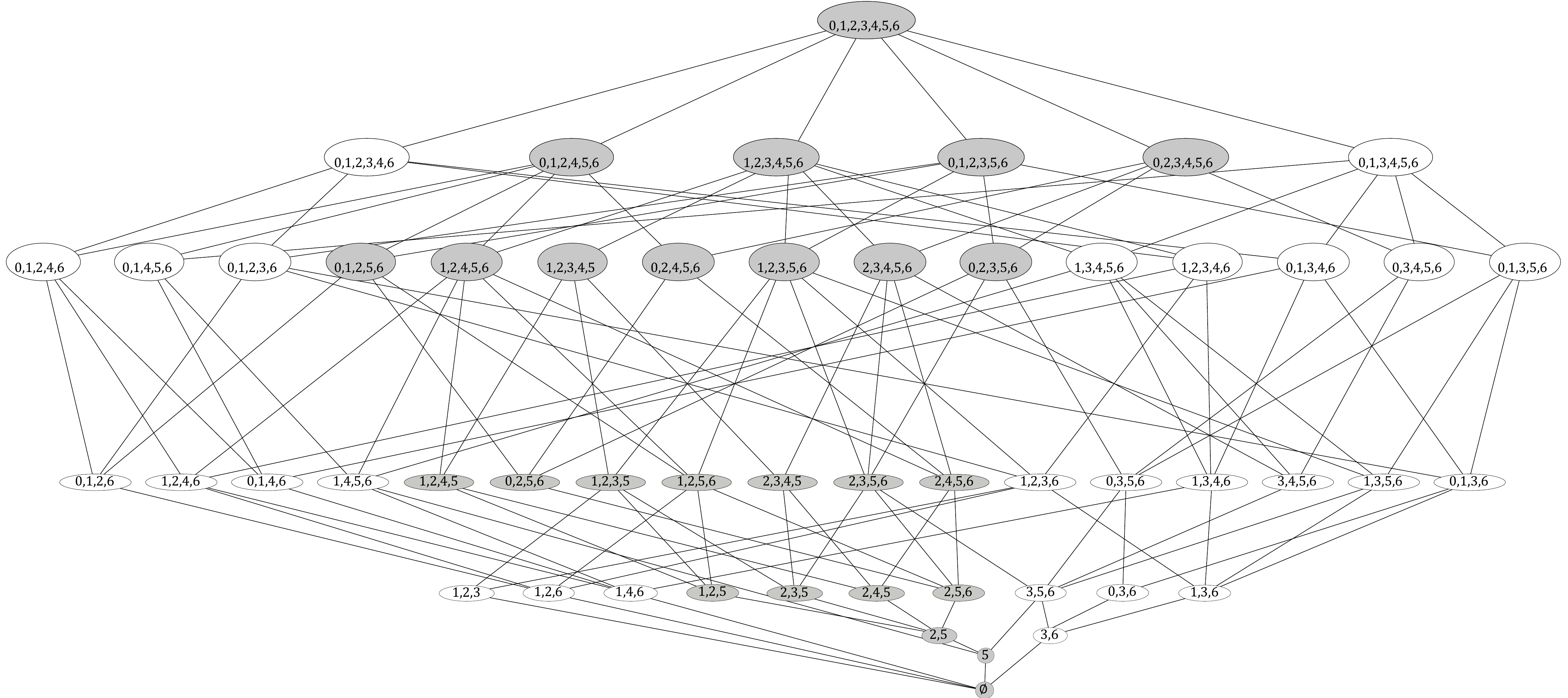}
\caption{$\mathcal{K}(\mathcal{R})$ and $\mathcal{A}(\mathcal{R})$. }
\label{fig:KRAR}
\end{figure}

\subsection{Closure operator and entailment}
Here we first summarize the basic relationship 
between Horn rules, entailments, and closure operators.
This is essentially one given by Dowling~\cite[Section 2]{Dowling93};
see~\cite{BM10, Wild14survey} from the view of implicational systems.
We then discuss entailments for antimatroids (or convex geometries).

We allow rule $(A,q)$ to be $q \in A$; 
such a rule accepts every subset, and is called {\em trivial}.
Thus a set of rules is a subset of $2^{Q} \times Q$, and is a binary relation between $2^Q$ and $Q$
(called an {\em implication relation} by Dowling~\cite{Dowling93}).
An {\em entailment} ({\em entail relation}) on $Q$ is 
a binary relation ${\cal R} \subseteq 2^{Q} \times Q$ satisfying
\begin{itemize}
	\item[(E1)] for all $A \subseteq Q$ and $q \in A$, it holds $A {\cal R} q$.
	\item[(E2)] for all nonempty $A,B \subseteq Q$ and $q \in Q$, 
	if $A{\cal R}b$ holds for all $b \in B$ and $B {\cal R} q$ holds, 
	then $A{\cal R}q$ holds.
\end{itemize}
Here we say ``{\em $A {\cal R} q$ holds}" if $(A,q) \in {\cal R}$.
In the terminology of implicational systems,
an entailment is a {\em full unit implicational system (full UIS)}~\cite{BM10}.

The following is a fundamental relationship between union-closed families and entailments;
recall that an implicate of a family ${\cal K}$ is 
a rule $(A,q)$ accepted by all members of ${\cal K}$, i.e., 
$A \cap X \neq \emptyset$ or $q \not \in X$ for all $X \in {\cal K}$. 
\begin{thm}[{Galois connection~\cite{KD1990,M1989}; see~\cite[Chapter 7]{FD2011}}]\label{thm:connection}
	\begin{itemize}
		\item[{\rm (1)}] For ${\cal K} \subseteq 2^{Q}$, 
		 the set ${\cal R}$ of all implicates of ${\cal K}$ is an entailment, and ${\cal K}({\cal R})$ is the unique minimal union-closed family containing ${\cal K}$.
		 In particular, if ${\cal K}$ is union-closed, then ${\cal K} = {\cal K}({\cal R})$. 
		 \item[{\rm (2)}] For  ${\cal R} \subseteq 2^{Q} \times Q$, the set of all implicates of ${\cal K}({\cal R})$ is the unique minimal entailment containing ${\cal R}$.
		 In particular, if ${\cal R}$ is an entailment, then ${\cal R}$ is equal to the set of all implicates of ${\cal K}({\cal R})$. 
	\end{itemize}
\end{thm}	
For a family ${\cal K}$, let ${\cal K}^*$ denote the {\em dual}  of ${\cal K}$ defined by
\[
{\cal K}^* := \{ Q \setminus X \mid X \in {\cal K} \}.
\]
The dual of a union-closed family is an intersection-closed family (or a {\em closure system}). 
For an intersection-closed family ${\cal J}$ on $Q$, 
we can define a map $\tau:2^{Q} \to 2^{Q}$, 
called the {\em closure operator}, 
by 
\begin{equation*}
\tau(X) := \bigcap \{ Y  \in {\cal J} \mid X \subseteq Y\} \quad (X \subseteq Q).
\end{equation*}
For a union-closed family ${\cal K}$ and 
the closure operator $\tau$ of dual ${\cal K}^*$, it obviously holds that
\begin{equation}\label{eqn:tau}
Q \setminus \tau (X) = (Q \setminus X)^{\circ} \quad (X \subseteq Q).
\end{equation}
We will often use 
the following relation between implicates of ${\cal K}$ and 
closure operator $\tau$ of ${\cal K}^*$, 
which was given by Dowling-M\"{u}ller~\cite{Dowling93,M1989} in the literature of KST.
\begin{lem}[{\cite[Proposition 4.4 (i)]{M1989}; 
		see \cite[Proposition 2.8]{Dowling93}}]\label{lem:imp_equiv}
	Let ${\cal K}$ be a union-closed family and let $\tau$ be the closure operator of the dual ${\cal K}^*$. For a rule $(A,q)$,  the following conditions are equivalent:
	\begin{itemize}
		\item[{\rm (i)}] $(A,q)$ is an implicate of ${\cal K}$.
		\item[{\rm (ii)}] $q \in \tau (A)$.
		\item[{\rm (iii)}] $q \not \in (Q \setminus A)^\circ$.
	\end{itemize}
\end{lem}
\begin{proof}
	The equivalence between (ii) and (iii) follows from (\ref{eqn:tau}).
	Condition (ii) is equivalent to: $q$ belongs to $\bigcap \{ Q \setminus X \mid  X \in {\cal K}: X \cap A =\emptyset \}$.
	This is further equivalent to: 
	every $X \in {\cal K}$ disjoint with $A$ satisfies $q \not \in X$. 
	This is nothing but condition (i).
\end{proof}

We next give a characterization of the entailment of an antimatroid.
As Koppen did in~\cite[p.142, (27)]{K1998},
such a characterization of ${\cal R}$ is directly obtained from  {\bf (Accessibility)} as:
\begin{description}
	\item For every subset $X \subseteq Q$ with $X \cap A \neq \emptyset$ or $q \not \in X$ for all $(A,q) \in {\cal R}$, there exists an element $x$ in $X$ such that $(X - x) \cap A \neq \emptyset$ or $q \not \in X - x$ for all $(A,q) \in {\cal R}$.
\end{description}
This is a rather cumbersome condition. 
We here provide another useful characterization, which is directly obtained 
from the anti-exchange property of convex geometries dual to antimatroids.

A {\em convex geometry}~\cite{EdelmanJamison85} is 
an intersection-closed family with its closure operator $\tau$ satisfying
the following {\em anti-exchange property} (AE):
\begin{description}
\item[(AE)] for $X\subseteq Q$ and distinct $y, z\in Q$, if $y, z\notin \tau(X)$ and $z\in\tau(X + y)$ then $y\notin \tau(X + z)$. 
\end{description}
It is well known that a family is a convex geometry 
if and only if its dual is an antimatroid~\cite[Theorem 1.3]{KLS1991}.
Therefore, by using Lemma~\ref{lem:imp_equiv} and Theorem~\ref{thm:connection}, 
the condition (AE) is translated into a condition for an entailment ${\cal R}$ 
to have an antimatroid ${\cal K}({\cal R}) ( = {\cal A}({\cal R}))$.
We will use the following slightly different characterization.    
\begin{prop}\label{prop:AE}
For an entailment ${\cal R}$ on $Q$, 
the family $\mathcal{K}({\cal R})$ is an antimatroid if and only if ${\cal R}$ satisfies:
\begin{description}
	\item[(AE$'$)] for $X\subseteq Q$ and distinct $y, z\in Q$, if 
	$(X + y){\cal R} z$ and $(X + z) {\cal R} y$ hold, 
	then $X {\cal R} y$ and $X {\cal R} z$ hold.
\end{description} 
\end{prop}
\begin{proof}
Observe that (AE$'$) implies (AE). 
We show the converse.
Suppose that entailment ${\cal R}$ satisfies (AE).
To show (AE$'$), suppose that $(X + y)\mathcal{R}z$ and $(X + z)\mathcal{R}y$ hold. 
By (AE), $X\mathcal{R}y$ or $X\mathcal{R}z$ holds.
By symmetry, we can assume that $X\mathcal{R}y$ holds. 
Then, by (E1), $X {\cal R} x$ holds for all $x \in X +  y$. 
By (E2) and $(X + y)\mathcal{R}z$, we have $X\mathcal{R}z$, and obtain (AE$'$). 
\end{proof}
Based on this characterization,  in Section~\ref{sec:algorithms}
we develop an algorithm to generate all implicates of ${\cal A}({\cal R})$.

\subsection{Korte-Lov\'{a}sz representation}\label{sec:circuit}

We here explain the relationship between 
the above construction of an antimatroid from Horn rules
and the construction of an antimatroid from rooted sets by Korte and Lov\'{a}sz \cite{KL1984}; 
see also \cite[Section III. 3]{KLS1991}. 
A {\em rooted set} $(C,r)$ is a pair of a subset $C \subseteq Q$ and an element $r$ in $C$, 
where $r$ is called the {\em root} of $(C,r)$.
For a given family ${\cal C}$ of rooted sets,  
Korte and Lov\'{a}sz  defined a family ${\cal L}({\cal C}) \subseteq 2^{Q}$ as follows:
A subset $X$ is a member of ${\cal L}({\cal C})$ if and only 
if there is an ordering $x_1,x_2,\ldots,x_k$ of the elements in $X$ 
such that $(C,x_i) \in {\cal C}$ implies $C \cap \{x_1,x_2,\ldots,x_{i-1}\} \neq \emptyset$ for each $i$.
In \cite{KN13}, a family (or a language) determined in this way 
is called a {\em transversal precedence structure}.

Korte and Lov\'{a}sz showed that ${\cal L}({\cal C})$ 
is always an antimatroid~\cite[Lemma 3.2]{KLS1991}.
In fact, ${\cal L}({\cal C})$ is equal to ${\cal A}({\cal R})$, 
provided a nontrivial rule $(A,q)$ is associated with a rooted set $(A + q, q)$.
The map $(A,q) \mapsto (A + q,q)$ 
is a bijection between the set of all nontrivial rules and the set of all rooted sets.
\begin{thm}\label{thm:rep_Horn_circuit}
For a set $\mathcal{R}$ of nontrivial rules, let ${\cal C}$ 
be the set of rooted sets defined by ${\cal C} := \{ (A + q, q)  \mid (A,q) \in {\cal R}\}$.
Then it holds $\mathcal{A}(\mathcal{R})=\mathcal{L}({\mathcal{C}})$. 
\end{thm}

\begin{proof}
	Pick $X$ from $\mathcal{L}({\mathcal{C}})$. 
	By definition, there is an ordering $x_1,x_2,\ldots,x_k$ of $X$ such that
    $(A, x_i) \in {\cal R}$ implies 
    $A \cap \{x_1,x_2,\ldots,x_{i-1}\} \neq \emptyset$.
    This means that $\{x_1,x_2,\ldots,x_i\} \in {\cal K}({\cal R})$ for $i=1,2,\ldots,k$.
	Thus $\{x_1,x_2,\ldots,x_i\}$ $(i=1,2,\ldots,k)$ form a tight path from $\emptyset$ to $X$ in ${\cal K}({\cal R})$, and $X \in {\cal A}({\cal R})$.
	Conversely, pick $X$ from ${\cal A}({\cal R})$.
	Then there is a tight path $\emptyset, \{x_{1}\}, \{x_{1}, x_{2}\}, \ldots , \{x_{1}, x_{2}\ldots , x_{k}\}=X$ in $\mathcal{K}(\mathcal{R})$.
    The ordering $x_1,x_2,\ldots,x_k$ of $X$ fulfills the definition of ${\cal L}({\cal C})$, hence  $X \in {\cal L}({\cal C})$. Thus we conclude that $\mathcal{A}(\mathcal{R})=\mathcal{L}({\mathcal{C}})$.
\end{proof}
Korte and Lov\'{a}sz introduced a natural class of rooted sets
determined by and determining an antimatroid.
Let ${\cal A}$ be an antimatroid,  ${\cal A}^*$ the convex geometry dual to ${\cal A}$, 
and $\tau$ the closure operator of ${\cal A}^*$.
A subset $X$ is said to be {\em free} if $\{X \cap K \mid K \in {\cal A}\}$ is equal to $2^X$.
A {\em circuit} is a subset $C$ such that $C$ is not free and every proper subset of $C$ is free\footnote{Any free subset is a subset of $Q^{\circ}$, 
and circuits are for the proper antimatroid on $Q^{\circ}$.}.
It is known in \cite[Lemma 3.2]{KLS1991} that
there is an unique element $r$, called the {\em root},  in a circuit $C$
such that $\tau(C) -  r \not \in {\cal A}^*$ 
and $\tau(C) - s \in {\cal A}^*$ for $s \in C - r$.
Now a circuit $C$ is regarded as a rooted set $(C,r)$ for the root $r$ of $C$.
A {\em critical circuit} is a rooted set $(C,r)$ 
such that $\tau(C) - r \not \in  {\cal A}^*$ 
and $\tau(C) - r - s \in  {\cal A}^*$ for $s \in C - r$. 
It is known in \cite[p.31]{KLS1991} that a critical circuit is indeed a circuit.
Circuits can determine the original antimatroid, 
and such circuits always contain all critical circuits, as follows.
\begin{thm}[{\cite{KL1984}; see \cite[Theorem 3.11]{KLS1991}}]\label{thm:circuit}
	Let ${\cal A}$ be an antimatroid, and let ${\cal S}$ be the set of critical circuits of ${\cal A}$.
	Then the following hold:
	\begin{itemize}
		\item[{\rm (1)}] ${\cal A} = {\cal L}({\cal S})$.
		\item[{\rm (2)}] For any family ${\cal C}$ of circuits of ${\cal A}$, 
		if ${\cal A} = {\cal L}({\cal C})$, then ${\cal S} \subseteq {\cal C}$.
	\end{itemize}
\end{thm}
We develop in Section~\ref{sec:critical_o2} an algorithm 
to construct the set ${\cal S}$ of all critical circuits of ${\cal A} = {\cal A}({\cal R})$  
from ${\cal R}$.
For this purpose, we use the following Boolean function theoretic characterization of circuits.
A nontrivial implicate $(A,q)$ of ${\cal A}$ is said to be {\em prime} 
if for every $a \in A$, the rule $(A - a, q)$ is not an implicate of ${\cal A}$.
This definition is consistent with prime implicates of the corresponding Horn CNF~\cite[Definition 1.21]{Booleanfunctions}.
The set of prime implicates is shown to be equal to 
the {\em canonical direct unit implicational basis} in~\cite{BM10}.
In the case of an antimatroid (convex geometry), 
Wild~\cite{Wild94} showed that this basis is exactly the set of  all circuits.
\begin{prop}[{implied by \cite[Theorem 15, Corollary 22]{BM10} and \cite[Corollary 13]{Wild94}}]\label{prop:Wild}
	Let  ${\cal A}$ be an antimatroid.
	\begin{itemize}
		\item[{\rm (1)}] For a circuit $(C,r)$ of ${\cal A}$, the rule $(C - r, r)$ is a prime implicate of ${\cal A}$.
		\item[{\rm (2)}] For a prime implicate $(A,q)$ of ${\cal A}$, 
		the rooted set $(A + q, q)$ 
		is a circuit of ${\cal A}$.
	\end{itemize}
\end{prop}
For completeness, we give a direct proof in Appendix.
It should be noted that
the circuit concept is extended to general closure spaces/union-closed families~\cite{KN13}, 
and an analogous characterization holds~\cite[Section 3.3]{Wild14survey}.

In particular, 
for any union-closed family ${\cal K}$, the set ${\cal R}$ of all prime implicates 
determines ${\cal K}$ as ${\cal K} = {\cal K}({\cal R})$. 
Thus we have the following.
\begin{cor}[{\cite{KL1984}; see~\cite[Lemma 2]{Dietrich87}}]\label{cor:A(R)=K(R)}
	Let ${\cal A}$ be an antimatroid,  ${\cal C}$ the set of all circuits in ${\cal A}$, and ${\cal R} := \{(C - r, r) \mid (C, r) \in {\cal C}\}$ the corresponding set of rules.
   Then it holds
   \[
   {\cal A} = {\cal A}({\cal R}) = {\cal K}({\cal R}).
   \]
\end{cor}
As remarked in \cite[p. 137]{Wild94} and \cite[Example 21]{KN13}, 
the set ${\cal R}^*$ of critical rules (circuits) of an antimatroid ${\cal A}$ is not necessarily 
an implicational basis of ${\cal A}$, i.e., 
${\cal A} \subseteq {\cal K}({\cal R}^*)$ possibly strict.

\section{Algorithms}\label{sec:algorithms}

\subsection{Membership and inference problems}

We give linear time algorithms for the membership and inference problems defined in the introduction. 
From now on, let $n$ denote the cardinality of the ground set $Q$. 
The size $l({\cal R}) := \sum_{(A,q)\in\mathcal{R}}(|A|+1)$ of input $\mathcal{R}$ is simply denoted by $l$. 
We tacitly assume that $l \geq n$.
We first provide a linear time algorithm to 
compute the maximum member $X^{\circ} \subseteq X$ in ${\cal A}({\cal R})$
from given $X \subseteq Q$ and ${\cal R}$. 
Linear time algorithms for the membership and inference problems 
are immediate consequences (via Lemma~\ref{lem:imp_equiv}). 

The idea of our algorithm is to trace a tight path on $\mathcal{K}(\mathcal{R})$ 
from the empty set. 
We show that a tight path to $X$, if it exists, is obtained 
by greedily adding elements from the empty set.
\begin{prop}\label{prop:path_maximal}
Let $\mathcal{K}$ be a  union-closed family on $Q$. For any subset $X$ of $Q$, the following statements are equivalent: 
\begin{itemize}
\item[{\rm (i)}] There exists a tight path from $\emptyset$ to $X$ in $\mathcal{K}$. 
\item[{\rm (ii)}] For any member $S$ in $\mathcal{K}$, if $S$ is a proper subset of $X$, 
there exists an element $x\in X\setminus S$ such that $S + x \in\mathcal{K}$. 
\end{itemize} 
\end{prop}

\begin{sloppypar}
\begin{proof}
It is obvious that (ii) implies (i). Indeed, according to (ii),
we can construct a tight path from $\emptyset$ to $X$ by adding elements of $X$.

We next show that (i) implies (ii). 
Suppose that $\mathcal{K}$ contains a tight path $\emptyset , \{q_{1}\}, \{q_{1}, q_{2}\}, \ldots , \{q_{1}, q_{2}, \ldots , q_{m}\}=X$. Let $S$ be a member of $\mathcal{K}$. 
Suppose that $S$ is a proper subset of $X$.
Take the minimum index $i$ with $q_{i}\notin S$. By the union-closedness of $\mathcal{K}$, we have $S + q_{i} = S\cup \{q_{1}, q_{2}, \ldots , q_{i}\} \in\mathcal{K}$. Thus $q_i$ is a required element. 
\end{proof}
\end{sloppypar}

According to Proposition \ref{prop:path_maximal}, 
we easily obtain the following algorithm to compute $X^{\circ}$. 
Starting with $S=\emptyset$, if there exists an element $x\in X\setminus S$ such that $S + x$ 
is again a member of $\mathcal{K}(\mathcal{R})$, then add $x$ to $S$, and repeat. 
If such an element $x$ does not exist, then the current $S$ is actually equal to $X^{\circ}$.
In each iteration, we can check whether $S + x$ is a member of $\mathcal{K}(\mathcal{R})$ in $O(l)$. 
Therefore the time complexity of this algorithm is $O(n^{2}l)$. 

We are going to improve this naive algorithm by searching the feasible continuations $x$ of the current $S$ efficiently. 
We first point out that it requires only a part of the input $\mathcal{R}$ to check whether $S + x$ belongs to $\mathcal{K}(\mathcal{R})$. 
Indeed, we do not have to care about the elements not in $X$. 
Hence, instead of $(A, q)$,
we may keep the rule $(A\cap X, q)$. 
If $q$ is not in $X$, then $(A,q)$ can be ignored. 
Also, if $(A, q)$ satisfies $A\cap S\neq \emptyset$, 
then $(A, q)$ accepts $S + x$ for any $x$, and can be ignored.
Summarizing, instead of $\mathcal{R}$, 
it suffices to keep, in each step, the following set of rules: 
\begin{equation}\nonumber
\mathcal{R}^{X}_{S}:=\{(A\cap X,q)\mid (A,q)\in\mathcal{R}: A\cap S=\emptyset, q\in X\setminus A\}. 
\end{equation}
The next lemma explains how to obtain $x$ with $S + x \in {\cal K}({\cal A})$ from $\mathcal{R}^{X}_{S}$.
\begin{lem}\label{lem:member_add}
Let $X$ be an arbitrary subset of $Q$, and $S$ a member of $\mathcal{K}(\mathcal{R})$ properly contained by $X$.
 For $x\in X\setminus S$, the subset $S + x$ belongs to $\mathcal{K}(\mathcal{R})$ if and only if $\mathcal{R}^{X}_{S}$ has no rule of the form~$(A,x)$. 
\end{lem}

\begin{proof}
	The if part is obvious from the above discussion.
	To show the only if part, let $(A,x)\in\mathcal{R}^{X}_{S}$. 
	Since $A\cap S=\emptyset$ (and $x\notin A$), 
$A$ has no intersection with $S + x$. Hence $S + x\notin \mathcal{K}(\mathcal{R})$. 
\end{proof}

Now we are ready to describe our algorithm. 
\begin{algo}[to compute $X^{\circ}$]\label{algo:Xcirc}\ 
\begin{algorithmic}[1]
\REQUIRE A set $\mathcal{R}$ of Horn rules and a subset $X$ of $Q$. 
\ENSURE $X^{\circ}$.
\STATE $S\leftarrow \emptyset$
\STATE $C:=\{x\in X\setminus S\mid \mathcal{R}^{X}_{S} \mbox{ has no rules of the form } (A, x)\}$
\IF{$C=\emptyset$}
	\RETURN $S$
\ENDIF
\STATE Choose $q \in C$, $S\leftarrow S + q$ (or $S \leftarrow S \cup C$ directly), 
update $\mathcal{R}^{X}_{S}$, and go to line 1
\end{algorithmic}
\end{algo}

We give a linear time implementation of this algorithm 
by using an appropriate data structure to maintain $\mathcal{R}^{X}_{S}$.
Suppose that
$\mathcal{R}=\{R_{1}, R_{2},\ldots , R_{m}\}$ and $R_i = (A_i, q_i)$ for $i=1,2,\ldots,m$.
In each iteration, 
we retain $\mathcal{R}^{X}_{S}$ by three kinds of lists $H_x$, $T_x$, and $E$. 
For each $x\in X\setminus S$, let $H_{x}$ be the set of indexes $i$ such that $A_{i}\cap S=\emptyset$ and $x\in A_{i}$, and $T_{x}$ be the set of indexes $j$ 
such that $x = q_j$ and $A_{j}\cap S=\emptyset$. 
Let $E$ be the set of elements $q\in X\setminus S$ such that $T_q$ is empty.
Here $T_x$ is kept by a doubly-linked list, and $H_x$ and $E$ are kept by a stack or queue.

The initialization is done by the following.
Look $(A_i,q_i)$ for $i=1,2,\ldots,m$.
Append index $i$ to $T_{q_i}$. 
We also keep a pointer from $i$ to ``$i$" in the list $T_{q_i}$. 
For each element $x \in A_i$, 
push index $i$ to $H_x$.
After that, 
$E$ is obtained by pushing elements $q$ with $T_q = \emptyset$. 
The total time is $O(l)$; recall $n \leq l$.
In each iteration, it holds $C=E$.
Pop $q$ from $C$, and add $q$ to $S$. 
The update $\mathcal{R}_{S}^{X}$ is as follows. 
Pop all indices $i$ from $H_{q}$, and remove $i$ from the list $T_{q_{i}}$ 
by tracing the pointer from $i$ (in constant time), 
since it now holds $A_{i}\cap S=\{q\}\neq\emptyset$. 
If $T_{q_{i}}$ gets empty, then we push $q_{i}$ to $E$. 
The computation time of the update is $O(\sum_{i \in H_q}|A_i|)$. 
Since each rule in $\mathcal{R}$ 
is referred at most once, the total time complexity is $O(l)$.
\begin{thm}\label{thm:closure_operator}
	For a set ${\cal R}$ of rules and a subset $X$, 
	the maximum member $X^{\circ} \subseteq X$ in ${\cal A}({\cal R})$ can be obtained in linear time.
\end{thm}
Dually speaking, 
the closure operator of the dual of ${\cal A}({\cal R})$ is computable in linear time.
Now a linear algorithm for membership problem (Theorem~\ref{thm:membership}) is immediate:  
\begin{algo}[to solve the membership problem]\label{algo:membership}\
	\begin{algorithmic}[1]
		\REQUIRE A set ${\cal R}$ of rules and a subset $X$.
		\ENSURE YES if $X \in {\cal A}({\cal R})$, and NO otherwise.
		\STATE Obtain $X^{\circ}$ by Algorithm~\ref{algo:Xcirc}.
		\IF{$X = X^{\circ}$} \RETURN YES   \quad \COMMENT{$X$ is a member of ${\cal A}({\cal R})$}.
		\ENDIF
		\RETURN NO  \quad \COMMENT{$X$ is not a member of ${\cal A}({\cal R})$}.  
	\end{algorithmic}	
\end{algo}
Recalling Lemma~\ref{lem:imp_equiv} for a characterization of implicates, 
we obtain a linear time algorithm for the inference problem (Theorem~\ref{thm:inference}) as follows.
\begin{algo}[to solve the inference problem]\label{algo:inference}\
	\begin{algorithmic}[1]
		\REQUIRE A set ${\cal R}$ of rules and a rule $(A,q)$.
		\ENSURE  YES if $(A,q)$ is an implicate of ${\cal A}({\cal R})$, and NO otherwise.
		\STATE Obtain $(Q\setminus A)^{\circ}$ by Algorithm~\ref{algo:Xcirc}.
		\IF{$q \notin (Q\setminus A)^{\circ}$}\RETURN YES \quad \COMMENT{$(A,q)$ is an implicate of ${\cal A}({\cal R})$}.
		\ENDIF
		\RETURN NO  \quad \COMMENT{$(A,q)$ is not an implicate of ${\cal A}({\cal R})$}.  
	\end{algorithmic}	
\end{algo}

\subsection{Generating all members of ${\cal A}({\cal R})$}\label{sec:enumeration}

As an application of the membership algorithm, 
we here provide a simple and efficient algorithm 
to enumerate all members in ${\cal A}({\cal R})$.
The efficiency of an enumeration algorithm
is measured by the time delay (interval) between two consecutive outputs.
Our enumeration algorithm is of $O(nl)$ time delay, and is designed 
on the basis of the standard technique of the {\em reverse search}~\cite{AF96}.

Let ${\cal A} = {\cal A}({\cal R})$  be an antimatroid given by a set ${\cal R}$ of rules.
Let $Q = \{1,2,\ldots,n\}$; we will use the natural ordering~$<$. 
Following terminologies in \cite{FD2011},
the {\em outer fringe} $X^{\cal O}$ and 
the {\em inner fringe} $X^{\cal I}$ of a member $X \in {\cal A}$ are defined by
\begin{eqnarray}\label{eqn:outer}
X^{\cal O} & := & \{x \in Q \setminus X \mid X + x  \in {\cal A}\}, \\
X^{\cal I} & := & \{x \in X \mid X - x  \in {\cal A}\}.
\end{eqnarray}
For a nonempty member $X$ in ${\cal A}$, let $\phi(X)$ be defined as the element $q$ in $X^{\cal I}$ such that 
$q$ is added in the last iteration of Algorithm~\ref{algo:Xcirc} 
for the input $X (= X^{\circ})$.
The map $X \mapsto \phi(X)$ is well-defined 
if the data structures of inputs ${\cal R}$ and $X$ are fixed 
(so that the smallest $q$ with respect to $<$ is chosen from $C$ in line 6 of Algorithm~\ref{algo:Xcirc}).
Let ${\cal F}(X)$ be the set of elements $x$ in $X^{\cal O}$ with $x = \phi(X + x)$. 
Let ${\cal T}$ be a directed graph on ${\cal A}$,  
where an edge from $X$ to $X'$ is given if and only if $X - \phi(X) = X'$.
Every nonempty member $X$ 
has exactly one edge leaving $X$.
Thus we obtain the following.
\begin{lem}
	${\cal T}$ is a spanning rooted tree with root $\emptyset$.
\end{lem}
Our algorithm is a depth first search on ${\cal T}$.
Starting at the root $X = \emptyset$,
if $X$ is labeled (or output), 
then we next label $X + x$ by choosing the smallest $x$ 
from ${\cal F}(X)$ with unlabeled $X + x$. 
If such an $x$ does not exist, then backtrack by computing $\phi(X)$.
An important point is that this can be done only by a local information at $X$.
\begin{algo}[to enumerate all members in ${\cal A}({\cal R})$]\label{algo:enumeration}\
	\begin{algorithmic}[1] 
 	\REQUIRE A set ${\cal R}$ of rules.
 	\ENSURE All members in ${\cal A}({\cal R})$.
	\STATE $X \leftarrow  \emptyset$.
	\STATE Output $X$.
	\STATE Compute ${\cal F}(X)$.
	\IF{${\cal F}(X)$ is empty}
	\STATE go to line 11.
	\ENDIF
	\STATE Choose the minimum element $y$ in ${\cal F}(X)$, $X \leftarrow X + y$, and go to line 2.
	\IF{$X = \emptyset$} 
	\STATE stop. 
	\ENDIF 
	\STATE $y \leftarrow \phi(X)$ and $X \leftarrow X - y$. \quad  \COMMENT {backtracking}     
	\IF{there is no element $x$ in ${\cal F}(X)$ with $x > y$}
	\STATE go to line 8.
	\ENDIF
	\STATE Choose the minimum element $x$ in ${\cal F}(X)$ with $x > y$, $X \leftarrow X +x$, and go to line 2.
	\end{algorithmic}
\end{algo}
We can compute $\phi(X)$ in $O(l)$ time  (Algorithm~\ref{algo:membership}), and
can compute ${\cal F}(X)$ in $O(nl )$ time by $n$ computations of $\phi$.
We estimate the time delay between consecutive output.
Suppose that $X$ is output at line 2.
Next ${\cal F}(X)$ is computed in $O(n l)$ time.
If the algorithm goes to line 7, then it goes to line 2, and hence the delay is $O(n l)$ time.
Suppose that the algorithm goes to line 11.
The backtracking loop (lines 11 to 15) iterates at most $n$ times.
We need not to compute $\phi$ and ${\cal F}$ in the backtracking.  
In lines 7 and 15, if $y$ and ${\cal F}(X)$ are pushed into a stack,
then  $\phi(X)$ and ${\cal F}(X)$ in lines 11 and 12 are popped off the stack.  
Thus the backtracking loop is conducted in $O(n^2)$ time.
Summarizing, we have the following.
\begin{thm}
	Algorithm~\ref{algo:enumeration} enumerates all members of ${\cal A}({\cal R})$ 
	with $O(n l)$ delay.
\end{thm}

\subsection{Computing critical circuits}\label{sec:critical_o2}

As an application of the inference algorithm, 
we here present a quadratic time algorithm to construct  
all critical rules (or circuits) of antimatroid ${\cal A}({\cal R})$ 
from ${\cal R}$.
Recall from Section~\ref{sec:circuit} that
a critical rule of an antimatroid ${\cal A}$ is a rule $(A,q)$ 
such that the rooted set $(A + q, q)$ is a critical circuit of ${\cal A}$.
Our algorithm is justified by the following lemma, 
which is a slight generalization of Theorem~\ref{thm:circuit}~(2).
\begin{lem}\label{lem:critical_in_rep}
Let ${\cal R}$ be a set of rules.
For a critical rule $(A, q)$ of $\mathcal{A}(\mathcal{R})$, there exists a rule $(A', q)$ in $\mathcal{R}$ with $A \subseteq A'$.
\end{lem}
\begin{proof}
	Let $(A,q)$ be a critical rule for ${\cal A} = {\cal A}({\cal R})$. 
	Let $\tau$ denote the closure operator of the dual ${\cal A}^*$ of ${\cal A}$.
	Let $X : = Q \setminus \tau (A + q)$.
	By definition, $X$ is a member of ${\cal A}$ but $X + q$ is not.
	Then $X+q$ is also not a member of ${\cal K}({\cal R})$.
	There necessarily exists some rule $(A',q)\in\mathcal{R}$ 
	 such that $A' \cap (X + q) =\emptyset$. Take $s \in A$ arbitrarily. 
	 Since $(A,q)$ is a critical rule, by definition we have $\tau(A+q)- q- s \in {\cal A}^*$ and 
	 $X + q + s\in\mathcal{A}$.
	 Hence $A' \cap (X + q + s)\neq\emptyset$. 
	 It follows that $s\in A'$. Thus we have $A \subseteq A'$. 
\end{proof}
In particular, the set of critical rules is 
the unique minimal expression of an antimatroid 
(among all possible Horn representations).
The following corollary sharpens Theorem~\ref{thm:circuit}, 
and coincides with it if ${\cal R}$ corresponds to a set of circuits of ${\cal A}({\cal R})$.
\begin{cor}
	Let ${\cal R}$ be a set of rules, and let ${\cal R}^*$ 
	be the set of critical rules of ${\cal A}({\cal R})$.
	Then it holds that ${\cal A}({\cal R}^*) =  {\cal A}({\cal R})$, 
	$|{\cal R}^*| \leq |{\cal R}|$, and 
	$l({\cal R^*}) \leq l({\cal R})$.
\end{cor}
In particular, the following inclusion holds:
\[
{\cal A}({\cal R}^*) = {\cal A}({\cal R}) \subseteq {\cal K}({\cal R}) \subseteq  {\cal K}({\cal R}^*).
\]
Notice again that the inclusions are strict in general, since 
that the set ${\cal R}^*$ of critical circuits (rules) is not necessarily 
an implicational basis of ${\cal A}({\cal R})$ or ${\cal K}({\cal R})$.

Now we present a quadratic time algorithm to obtain this unique minimal expression.
\begin{algo}[to compute critical rules]\label{algo:critical}\
\begin{algorithmic}[1]
\REQUIRE A set $\mathcal{R}$ of rules.
\ENSURE All critical rules.
\FORALL{$(A,q)\in\mathcal{R}$}	
	\FORALL{$a\in A$}
		\IF{$(A - a, q)$ is an implicate of $\mathcal{A}(\mathcal{R})$}
			\STATE $\mathcal{R}\leftarrow \mathcal{R} - (A, q) + (A - a, q)$
			\STATE $A\leftarrow A - a$
		\ENDIF
	\ENDFOR
	\IF{$(A, q)$ is an implicate of $\mathcal{A}(\mathcal{R} - (A, q))$}
		\STATE $\mathcal{R}\leftarrow \mathcal{R} - (A, q)$
	\ENDIF
\ENDFOR
\RETURN $\mathcal{R}$
\end{algorithmic}
\end{algo}
To implement {\bf for all} in line 1 (resp. line 2) by a usual computer language,
we index ${\cal R}$ as ${\cal R} = \{(A_i, q_i)\}_{i=1}^m$ (resp. $A$ as $A = \{a_j\}_{j=1}^k$)
and use for-loop from $i =1$ to $m$ (resp. $j=1$ to $k$).
We call the inference algorithm (Algorithm~\ref{algo:inference}) in lines 3 and 8.
\begin{thm}
Given a set $\mathcal{R}$ of rules, 
Algorithm \ref{algo:critical} computes all critical rules of $\mathcal{A}(\mathcal{R})$ in time $O(l^{2})$. 
\end{thm}

\begin{proof}
	The total time of the algorithm is $O(l^2)$ 
	since the algorithm calls the inference algorithm $l$ times.
	We show that the output is actually the set of critical rules.
	Notice that ${\cal A}({\cal R})$ does not change in each step.	
	Indeed, if $(A - a, q)$ is an implicate of ${\cal A}({\cal R})$,
	then ${\cal A}({\cal R}) \subseteq {\cal K}({\cal R} - (A,q) + (A -a, q)) \subseteq {\cal K}({\cal R})$ holds, implying ${\cal A}({\cal R}) = {\cal A}({\cal R} - (A,q) + (A -a, q))$ by maximality (Theorem~\ref{thm:uniqueness}).
	Similarly, if $(A, q)$ is an implicate of ${\cal A}({\cal R} - (A,q))$,
	then ${\cal A}({\cal R}) \subseteq {\cal A}({\cal R} - (A,q)) = {\cal A}({\cal R} - (A,q)) \cap {\cal K}(\{(A,q)\}) \subseteq {\cal K}({\cal R})$, 
	implying ${\cal A}({\cal R}) = {\cal A}({\cal R} - (A,q))$.
	
	In lines 3 to 6, the rule $(A,q)$ becomes prime.
	Indeed, suppose to the contrary 
	that  $(A - a, q)$ is an implicate of ${\cal A}({\cal R})$ for some $a \in A$.
	Then $(A' - a, q)$ is not an implicate for $A \subseteq A'$, 
	where $A'$ is equal to $A$ at the step of $a$ chosen.
	However, by $A-a \subseteq A' - a$ and (E2), 
	the rule $(A - a, q)$ cannot be an implicate, a contradiction.
	Therefore, the output ${\cal R}$ consists of prime implicates.
	By lines 8 and 9, the set ${\cal R}$ becomes minimal.
	By Theorem~\ref{thm:circuit}, this must be equal to the set of critical rules.  
\end{proof}

By using Algorithm~\ref{algo:critical}, 
we can efficiently check whether two given sets of rules define the same antimatroid.
Indeed, make both sets of rules critical, and compare them; 
they are equal if and only if they define the same antimatroid.
\begin{cor}
 Given two sets ${\cal R}$ and ${\cal R}'$ of rules, 
 we can determine whether ${\cal A}({\cal R}) = {\cal A}({\cal R}')$ in time $O(l^{2})$ 
 with $l = \max \{l({\cal R}), l({\cal R'})\}$.
\end{cor}

\subsection{Generating all nontrivial implicates of ${\cal A}({\cal R})$}

It is a natural problem to construct a superset $\mathcal{R}'$ of given ${\cal R}$ 
that satisfies $\mathcal{K}(\mathcal{R}')=\mathcal{A}(\mathcal{R})$. 
We do not know a polynomial time algorithm to construct such a set ${\cal R}'$; 
see Section~\ref{sec:concluding} for further discussion.
We here provide a simple algorithm to a related problem 
of generating all implicates of $\mathcal{A}(\mathcal{R})$ from ${\cal R}$.
As was seen in Theorem~\ref{thm:connection},  
the set ${\cal E}$ of all implicates of ${\cal A}({\cal R})$ 
obviously satisfies ${\cal K}({\cal E}) = {\cal A}({\cal R})$.
Our algorithm sharpens Dietrich's construction~\cite[Proposition 8]{Dietrich87} of 
all circuits from critical circuits, and 
may be comparable with the {\em resolution principle} 
(or the {\em consensus procedure}) 
in Boolean function theory, 
where the resolution is used to generate (prime) implicates of a Horn formula
(or ${\cal K}({\cal R})$); see \cite[Chapter 6]{Booleanfunctions}.

Our algorithm is based on the following variation of Proposition~\ref{prop:AE}.
\begin{lem}\label{lem:amres}
For a set ${\cal R}$ of rules, 
it holds $\mathcal{K}(\mathcal{R})=\mathcal{A}(\mathcal{R})$ if and only if 
for every pair of nontrivial implicates $(A,q)$,$(A', q')$  of $\mathcal{K}(\mathcal{R})$, 
both $((A\cup A') - q - q', q)$ and $((A\cup A') - q - q', q')$ are implicates of $\mathcal{K}(\mathcal{R})$. 
\end{lem}

\begin{proof}
The condition $\mathcal{K}(\mathcal{R})=\mathcal{A}(\mathcal{R})$ is equivalent to 
that ${\cal K}({\cal R})$ itself is an antimatroid.
By Proposition \ref{prop:AE}, it suffices to 
show that the condition of the statement is equivalent to (AE$'$). 

(Only-if part).
Here $((A \cup A') - q', q)$ and $((A \cup A') - q, q')$ 
are also implicates (by (E1), (E2)).
By (AE$'$) with $X=(A\cup A') - q - q'$, $x = q$, and $y = q'$, 
both $((A\cup A') - q - q', q)$ and $((A\cup A') - q - q', q')$ are implicates. 

(If part).
Notice that (AE$'$) trivially holds if $z \in X$ or $y \in X$.
Nontrivial cases of (AE$'$) follow from 
letting $A = X + y$, $A' = X + z$,  $y = q'$, and $z = q$.
\end{proof}
For two rules $(A,q)$ and $(A', q')$, define a rule $R(A,q;A',q')$ by
\[
R(A,q;A',q') := ((A \cup A') - q -q', q).
\]\
Consider the following simple procedure:
\begin{algo}[to generate all nontrivial implicates of ${\cal A}({\cal R})$]\label{algo:resolution}\ 
\begin{algorithmic}[1]
\REQUIRE A set $\mathcal{R}_0$ of (nontrivial) rules. 
\ENSURE All nontrivial implicates.
\STATE $\mathcal{R} \leftarrow \mathcal{R}_0$
\IF {there exist nontrivial rules $(A,q), (A',q')\in\mathcal{R}$ such that
$R(A,q;A',q')$ or $R(A',q';A,q)$ does not belong to ${\cal R}$} 
\STATE $
{\cal R} \leftarrow {\cal R} + R(A,q;A',q') + R(A',q';A,q)
$, 
\STATE and go to line 2. \ENDIF
\RETURN ${\cal R}$.
\end{algorithmic}
\end{algo}
According to the analogy of the resolution, the procedure in line 3 
is called an {\em antimatroidal resolution}.
The following is a sharpening of \cite[Proposition 8]{Dietrich87}.
\begin{thm}~\label{thm:amres}
Algorithm~\ref{algo:resolution} 
computes all nontrivial implicates of ${\cal A}({\cal R}_0)$.
In particular, it holds ${\cal K}({\cal R}) = {\cal A}({\cal R}_0)$.
\end{thm}
\begin{proof}
	By Lemma~\ref{lem:amres}, the output ${\cal R} (\supseteq {\cal R}_0)$ 
	consists of implicates of ${\cal A}({\cal R}_0)$, and
	it holds ${\cal A}({\cal R}) = {\cal A}({\cal R}_0)$.
	Let ${\cal O}$ be the set of all trivial rules (implicates).
	Obviously ${\cal K}({\cal R}) = {\cal K}({\cal R} \cup {\cal O})$ and ${\cal A}({\cal R}) = {\cal A}({\cal R} \cup {\cal O})$. 
	We claim that ${\cal R} \cup {\cal O}$ is an entailment.
	 Suppose that this is true.
	 Then the entailment ${\cal R} \cup {\cal O}$
	 is the set of all implicates of ${\cal K}({\cal R} \cup {\cal O}) = {\cal K}({\cal R})$ by Theorem~\ref{thm:connection}, and satisfies the condition of Lemma~\ref{lem:amres} 
	 since ${\cal R}$ is closed under antimatroidal resolutions.
	 Hence ${\cal K}({\cal R} \cup {\cal O}) = {\cal A}({\cal R} \cup {\cal O}) = {\cal A}({\cal R}) = {\cal A}({\cal R}_0)$, and
	 ${\cal R}$ is the set of all nontrivial implicates of ${\cal A}({\cal R}_0)$.

	Thus it suffices to show that ${\cal R} \cup {\cal O}$ satisfies (E2) 
	(since (E1) follows from (E2) and trivial implicates).
Let $(A, b_{1})$, $(A, b_{2}), \ldots , (A, b_{k})$ and $(B,q)$ be rules in ${\cal R} \cup {\cal O}$ 
with  $B=\{b_{1}, b_{2}, \ldots , b_{k}\}$.
We show that $(A,q) \in {\cal R} \cup {\cal O}$.
We may assume that $q \notin A \cup B$ (otherwise $(A,q) \in {\cal R} \cup {\cal O}$).
Suppose that $(A, b_{1}),\ldots , (A, b_{k'}) \in {\cal R}$ and 
$(A, b_{k'+1}), \ldots , (A, b_{k}) \in {\cal O}$ for some $k'\in \{1,2,\ldots,k\}$.
We show by induction that for $i=1,\ldots , k'$, 
one can deduce $(A\cup (B\setminus \{b_{1}, \ldots , b_{i}\}), q)$ by antimatroidal resolutions, 
i.e., it holds $(A\cup (B\setminus \{b_{1}, \ldots , b_{i}\}), q) \in {\cal R}$.
The case $i=k'$ is the desired claim (since $A\cup (B\setminus \{b_{1}, \ldots , b_{i}\}) = A$). 
First, nontrivial rule $(A\cup (B\setminus \{b_{1}\}),q)$ is deduced from nontrivial $(A, b_{1})$ and $(B, q)$. 
Next, assume that $(A\cup(B\setminus \{b_{1}, \ldots , b_{i}\}), q)$ was deduced for some $i$ with $2\leq i \leq k'-1$ by antimatroidal resolutions. 
Then $(A\cup (B\setminus \{b_{1},\ldots , b_{i+1}\}),q)$ 
is deduced from $(A, b_{i+1})$ and $(A\cup(B\setminus \{b_{1}, \ldots , b_{i}\}), q)$. 
Thus $(A,q) \in {\cal R} \cup {\cal O}$ as required.
\end{proof}
\renewcommand{\proofname}{Proof}

\section{Application to educational systems}\label{sec:application}
In this section, we mention possible applications 
of our results to the design of computer-aided educational systems.
As mentioned in the introduction, 
an antimatroid is used as a mathematical model of the space of knowledge states of learners, 
and is called a learning space in the literature of Knowledge Space Theory (KST).
The ground set $Q$ is a set of questions in a certain domain, 
and the knowledge state of a learner is associated with a subset $X \subseteq Q$
which he/she answers correctly.
The collection of all possible knowledge states forms a family ${\cal L}$ of subsets of $Q$.
A KST-based educational system gives questions to a learner, estimates his/her knowledge state $X \in {\cal L}$ 
according to the answers, and poses questions for the subjects that he/she can acquire next.
If the state of the learner reaches $Q$, then it might be said that 
the learner masters all subjects in the domain.
The hypothesis that 
$\mathcal{L}$ is an antimatroid (a learning space) is reasonable as well as useful 
in the above learning process. 
Indeed, for the state $X$ of a learner, the outer fringe $X^{\cal O}$ (defined in (\ref{eqn:outer}) in Section~\ref{sec:enumeration})
is always nonempty, provided ${\cal L}$ is an antimatroid. 
Therefore the system naturally chooses a question $q$ from $X^{\cal O}$ and poses $q$ to the learner.

To realize the above learning process, 
the educational system needs to know, in advance, the space ${\cal L}$ of knowledge states.
The space ${\cal L}$ is constructed with the help of a human expert (teacher) of the domain of the questions.
It is practically impossible to ask to the expert whether $X$ is a state in ${\cal L}$ 
for all subsets $X \subseteq Q$, 
because this needs a huge number $2^{|Q|}$ of queries. 
Koppen~\cite{K1993}, Koppen and Doignon~\cite{KD1990}, and Dowling-M\"{u}ller~\cite{Dowling93,M1989} 
introduced an alternative procedure to construct ${\cal K}$ by 
querying rules $(A_1,q_1), (A_2,q_2),\ldots$ to the expert, 
where the query $(A,q)$ means: 
\begin{description}
	\item[{\rm (Q$_{A,q}$)}] ``Does a learner fail the question $q$, provided he/she fails every question in $A$ ?"
\end{description}
Then the space of knowledge states is estimated as ${\cal K}({\cal P})$ 
for the set ${\cal P}$ of queries to which the expert said ``yes".
The point is that it is not necessary to ask all queries, thanks to the inference rules (E1) (E2). 
Let ${\cal P}$ and ${\cal N}$ be 
the sets of the queries to which the expert said ``yes" and ``no", respectively.
If $(A,q)$ is an implicate of ${\cal K}({\cal P})$, 
then the query $(A,q)$ is automatically determined to be `yes' at this moment.
Such a query $(A,q)$ is called a {\em positive inference} of ${\cal P}$.
Let ${\cal P}^{\star}$ denote the set of all positive inferences of ${\cal P}$. 
Similarly, there are queries automatically determined to be `no' 
by the following rules obtained by (E1) and (E2):
\begin{description}
\item[{\rm (NI-1)}] If $(A,p)$ is `yes', $(A,q)$ is `no', and $(A+p,b)$ is `yes' for all $b \in B$, then  $(B,q)$ is `no'.
\item[{\rm (NI-2)}] If $(A,p)$ is `yes', $(B,p)$ is `no', and $(B+q,a)$ is `yes' for all $a \in A$, then  $(B,q)$ is `no'.
\item[{\rm (NI-3)}]  If $(A,p)$ is `no', $(B+q,p)$ is `yes', and $(A,b)$ is `yes' for all $b \in B$, then $(B,q)$ is `no'.
\end{description}
A {\em negative inference} of ${\cal P}$ and ${\cal N}$ is a `no' query $(B,q)$ 
obtained by repeated applications of~(NI-1), (NI-2) and (NI-3).
Let ${\cal N}^{\star}$ denote the set of all negative inferences of ${\cal P}$ and ${\cal N}$.
Dowling~\cite{Dowling93} suggests a useful characterization of negative inferences: 
\begin{description}
	\item[{\rm (NI$'$)}] 
	$(A,q) \in {\cal N}^{\star}$ if and only if some query in ${\cal N}$ 
	is an implicate of ${\cal K}({\cal P} + (A,q))$. 
\end{description}
Thus there are three types of queries in each step: positive inferences, negative inferences, 
and other queries that are called {\em undetermined}. 
Notice again that positive inferences and negative inferences are redundant to be asked.

The QUERY algorithm updates positive and negative inferences
by using the inference rules, once the expert returns the answer. 
The algorithm next chooses and gives an undetermined query to the expert.
Koppen~\cite{K1993} suggested the following selection rule of queries.
In the first stage, queries are of form $(\{p\},q)$.
In the second stage, queries are of form $(\{p,p'\},q)$.
In the $i$-th stage, queries are of form $(A,q)$ with $|A| = i$. 
He also suggested a stopping criterion 
that after the $i$-th stage there is no undetermined query $(A,q)$ with $|A| = i+1$.
Positive and negative inferences are collected in a table, and are used 
by the selection of queries asked to the expert.
There are several selection rules of queries~\cite[Section 15.2.9]{FD2011}.
The resulting space ${\cal K}({\cal P})$ of knowledge states 
is constructed from the table; see~\cite[Section 15.2]{FD2011}.

Dowling~\cite{Dowling93} developed 
a sophisticated version of the QUERY algorithm.
Instead of keeping all positive and negative inferences, 
her algorithm keeps the base of current ${\cal K}({\cal P})$
and the set $m({\cal N}^{\star})$ of maximal negative inferences.
A {\em maximal} negative inference is a negative inference $(A,q)$ with the property that  
$(A',q)$ is not a negative inference for every $A' \supset A$.
Then a query $(B,q)$ is a negative inference if and only if $B \subseteq A$ 
for some maximal negative inference $(A,q)$. 
Thus the set ${\cal N}^{\star}$ of negative inferences is manageable 
by the set $m({\cal N}^{\star})$ of maximal negative inferences.
Recall the the base ${\cal B}$ of ${\cal K}({\cal P})$ 
is the set of members that cannot be the union of other members of ${\cal K}({\cal P})$.
The inference problem (i.e., checking whether $(A,q) \in {\cal P}^{\star}$) 
can be easily solved by the base; see~\cite[Proposition 3.2]{Dowling93}.
Dowling gave explicit formulas of updating ${\cal B}$ and $m({\cal N}^{\star})$,
once the expert returns an answer~\cite[Theorems 4.1, 4.2, 4.3]{Dowling93}.
All states of
the resulting space ${\cal K}({\cal P})$ are 
efficiently generated from~${\cal B}$; see~\cite{Dowling93a}.

\paragraph{Revised QUERY algorithm.}
The QUERY algorithm was designed for the case where the target space
is assumed to be a union-closed family (a knowledge space).
Therefore the output ${\cal K}$ is not necessarily an antimatroid.
We present a simple revision of the QUERY algorithm to output an antimatroid.
Our revision is obtained by replacing ${\cal K}({\cal P})$ with ${\cal A}({\cal P})$, 
and is understood as a concrete realization of Doignon's adjusted QUERY algorithm~\cite{Doignon14ICFCA}.

As above, let ${\cal P}$ and ${\cal N}$ 
denote the sets of queries to which the expert said ``yes" and ``no", respectively.
We can naturally define positive/negative inferences for the revised QUERY.
A {\em strong positive inference} of ${\cal P}$ is an implicate of ${\cal A}({\cal P})$, 
and a {\em strong negative inference} of ${\cal P}$ and ${\cal N}$ is 
a query $(A,q)$ such that some query in ${\cal N}$ is 
an implicate of ${\cal A}({\cal P} + (A,q))$.
Let ${\cal P}^{\star \star}$ and ${\cal N}^{\star \star}$ denote the sets of all strong positive and negative inferences, respectively.
We remark
\[
{\cal P}^{\star} \subseteq {\cal P}^{\star \star},\ 
{\cal N}^{\star} \subseteq {\cal N}^{\star \star}
\]
by ${\cal A}({\cal P}) \subseteq {\cal K}({\cal P})$ and 
${\cal A}({\cal P} + (A,q)) \subseteq {\cal K}({\cal P} + (A,q))$.
Strong positive or negative inferences are redundant to be asked, since 
the estimated space is now ${\cal A}({\cal P})$.
The revision of QUERY is as follows.
\begin{description}
	\item[Revised QUERY algorithm] 
	\item[0:] Let ${\cal P} = {\cal P}^{\star \star} = {\cal N} = {\cal N}^{\star \star}:= \emptyset$
	\item[1:] Choose a query 
	$(A,q) \not \in {\cal P}^{\star \star} \cup {\cal N}^{\star \star}$ 
	(with smallest $|A|$), and ask the question (Q$_{A,q}$) to the expert.
	\item[2:] If the answer is ``yes", then add $(A,q)$ to ${\cal P}$
	(, and make ${\cal P}$ critical). 	
	If the answer is ``no", then add $(A,q)$ to ${\cal N}$.
	\item[3:] Update ${\cal P}^{\star \star}$ and ${\cal N}^{\star \star}$.  
	\item[4:] If a stopping criterion is fulfilled, then output ${\cal P}$ (or ${\cal A}({\cal P})$); stop. Otherwise go to step 1. 
\end{description}
Let us look details of this algorithm. 
In step 3, we can use  Algorithm~\ref{algo:resolution} to generate all strong positive inferences, i.e., 
all implicates of ${\cal A}({\cal P})$. 
In the practical situation where the storage for queries is limited (or $Q$ is large), 
we can use the inference algorithm (Algorithm~\ref{algo:inference}) 
to generate strong positive inferences $(A,q)$ with a specified size of $|A|$.
Similarly, we can generate strong negative inferences 
by applying the inference algorithm to queries in ${\cal N}$ for ${\cal A}({\cal P} + (A,q))$.
Koppen's termination criterion and other selection rules of queries are implementable.
By the use of Algorithm~\ref{algo:critical}, 
we may keep ${\cal P}$ compact.
All states in ${\cal A}({\cal P})$ can be efficiently generated by Algorithm~\ref{algo:enumeration}.

In the case where the output is required to be a proper antimatroid, 
we may modify step 2 as: 
If the answer is ``yes" and $Q \in {\cal A}({\cal P} + (A,q))$, then add $(A,q)$ to ${\cal P}$. 
Actually Doignon's adjusted QUERY adopts this rule.
The condition $Q \in {\cal A}({\cal P} + (A,q))$ can be 
checked by the membership algorithm (Algorithm~\ref{algo:membership}).

In the idealistic case where the answers of the expert correctly follows 
his/her latent antimatroid ${\cal L}$, by Theorem~\ref{thm:uniqueness} it always holds
\[
{\cal L} \subseteq {\cal A}({\cal P}) \subseteq {\cal K}({\cal P}).
\]
Therefore the resulting ${\cal A}({\cal P})$ might be 
a reasonable outer approximation of the true antimatroid ${\cal L}$.
Also sufficiently many queries uncover
${\cal L}$ as ${\cal L} = {\cal A}({\cal P})$; see \cite[Proposition 13]{Doignon14ICFCA}. 
\begin{rem}
	J.-P. Doignon asked what about the revised QUERY is applied 
	to the expert following a union-closed family ${\cal L}$.
	Also, in this case,  ${\cal L} \subseteq {\cal K}({\cal P})$ holds throughout iterations.
	Therefore, by maximality (Theorem~\ref{thm:uniqueness}), 
	$\check{\cal L} \subseteq {\cal A}({\cal P})$ holds.
	We could not guarantee that 
	the revised QUERY reaches $\check{\cal L} = {\cal A}({\cal P})$.
	The reason is the existence of a query in ${\cal N}$ that is `no' for ${\cal L}$ but is `yes' for $\check{\cal L}$. Consequently,
	even if a query $(A,q)$ satisfies
	$\check{\cal L} \subseteq  {\cal A}({\cal P} + (A,q)) \subset {\cal A}({\cal P})$,
	the query $(A,q)$ may fall into ${\cal N}^{\star \star}$ and is never posed. 
	
	This is not the case if a query is chosen outside ${\cal P}^{\star \star} \cup {\cal N}^{\star}$ in each iteration.
%
	In fact, if ${\cal A}({\cal P}) \setminus \check{\cal L} \neq \emptyset$, there is 
	$(A,q) \not \in {\cal P}^{\star \star} \cup {\cal N}^{\star}$ such that $(A,q)$ is an implicate of ${\cal L}$ and
	$\check{\cal L} \subseteq {\cal A}({\cal P} + (A,q)) \subset {\cal A}({\cal P})$;
	such a query is posed to decrease ${\cal A}({\cal P})$.
	To see this, consider minimal $X \in {\cal A}({\cal P}) \setminus \check{\cal L}$,  
	and consider $X^{\circ}$ with respect to $\check{\cal L}$. By the minimality and Proposition~\ref{prop:path_maximal}, 
	it must hold $|X \setminus X^{\circ}| = 1$.
	Therefore $X$ is not in ${\cal L}$, 
	and there is an implicate $(A,q)$ 
	of ${\cal L}$ (and of $\check{\cal L}$) not accepting $X (\in {\cal A}({\cal P}))$. 
	Obviously $(A,q) \not \in {\cal P}^{\star \star}$. Also $(A,q) \not \in {\cal N}^{\star}$
	by ${\cal L} \subseteq {\cal K}({\cal P} + (A,q))$. 
	Thus $(A,q)$ is a desired query.
	\end{rem}
\begin{rem}
	Our revision can incorporate 
	Dowling's update of maximal negative inferences. 
	Her formulas~\cite[Theorems 4.2, 4.3]{Dowling93}　
	only involve the closure operator of (dual of) ${\cal K}({\cal P})$.
	Thus the desired update formulas are obtained simply
	by changing the closure operator of ${\cal K}({\cal P})$ to the closure operator of ${\cal A}({\cal P})$.
	By Theorem~\ref{thm:closure_operator}, the closure operator of ${\cal A}({\cal P})$ is efficiently computable.
	On the other hand, the base update~\cite[Theorems 4.1]{Dowling93}　
	seems not to be adapted directly to ${\cal A}({\cal P})$.
	This issue is left to future work.
\end{rem}
 \begin{ex}
 	We give one small but instructive example.
 	Let $Q = \{0,1,2,3\}$ and ${\cal R}$ consist of two rules $(\{0,2\},1)$ and $(\{1,3\},0)$. 
 	The goal is to identify ${\cal L} := {\cal A}({\cal R})$ by QUERY algorithms.
 	Queries are examined, in the same order, for both original and revised QUERY algorithms.
 	Algorithms terminates if 
 	${\cal L}= {\cal K}({\cal P})$ for original and ${\cal L}= {\cal A}({\cal P})$ for revised.
 	Table~\ref{tab:ex} describes the behavior of two algorithms.
 	\begin{table}
 		\caption{Behavior of original and revised QUERY algorithms}
 		\begin{center}
 			\begin{tabular}{|c||c c|}
 				\hline
 				query &   original   & revised \\
 				\hline
 				$(A,q)$: $|A| \leq 1$ & posed:NO & posed:NO \\
 				$(\{0, 1 \},2)$ & posed:NO & posed:NO \\
 				$(\{0, 1 \},3)$ & posed:NO & posed:NO \\ 
 				$(\{0, 2 \},1)$ & posed:YES & posed:YES  \\
 				$(\{0, 2 \},3)$ & posed:NO &  posed:NO \\
 				$(\{0, 3 \},1)$ & posed:NO &  posed:NO\\
 				$(\{0, 3 \},2)$ & negainf & negainf \\
 				$(\{1, 2 \},0)$ & posed:NO &  negainf  \\
 				$(\{1, 2 \},3)$ & negainf & negainf  \\
 				$(\{1, 3 \},0)$ & posed:YES & posed:YES  \\
 				$(\{1, 3 \},2)$ & posed:NO  &   \\
 				$(\{2, 3 \},0)$ & posed:YES & \\
 				\hline
 			\end{tabular}\label{tab:ex}
 		\end{center}
 	\end{table}
 	The first column indicates queries, which are examined from top to down, 
 	and the second and third columns indicate 
 	the actions of the original and revised QUERY, respectively. 
 	Here posed:YES (resp. posed:NO) means that the query 
 	in the same row  is posed to the expert and the answer is ``yes" (resp. ``no"), 
 	and negainf means that 
 	the query is a (strong) negative inference (in revised QUERY) and is not posed.
 	The first $16$ (nontrivial) queries of form $(A,q)$ with $|A| \leq 1$ are posed for both original and revised, and the answers are all ``no''.
 	In total, 
 	the original QUERY posed $25$ queries to the expert. 
 	The revised QUERY posed  $22$ queries, and finished two queries earlier 
 	than the original QUERY, which is caused by ${\cal L} = {\cal A}({\cal R})  \subset {\cal K}({\cal R})$. 
 	Query $(\{1,2\}, 0)$ is not a negative inference but a strong negative inference, and hence is not posed in the revised QUERY. 
 \end{ex}

\paragraph{Preliminary experimental results.}
We conducted preliminary computer experiments to investigate how
the revision contributes to the reduction of the number of queries asked.
We prepare, in computer, a target antimatroid ${\cal L}$
and an idealistic expert who answers query $(A,q)$ correctly.
Namely the expert answers ``yes" if  $(A,q)$ is an implicate of ${\cal L}$, and ``no" otherwise. 
The goal is to identify ${\cal L}$ by queries. 
We compare two QUERY algorithms.
The first algorithm is (a simpler version of ) the original QUERY algorithm. 
In each step, the algorithm poses a query 
$(A,q) \not \in {\cal P}^{\star} \cup {\cal N}^{\star}$, 
where  ${\cal P}$ and ${\cal N}$ are 
the sets of `yes' queries and `no' queries, respectively, obtained so far. 
The algorithm terminates when ${\cal K}({\cal P}) = {\cal L}$.
The second algorithm is a simpler version of our revised QUERY algorithm, 
which is obtained from the first one  
by replacing ${\cal P}^{\star} \cup {\cal N}^{\star}$ 
with ${\cal P}^{\star \star} \cup {\cal N}^{\star \star}$
and replacing the termination criterion ${\cal K}({\cal P}) = {\cal L}$
with ${\cal A}({\cal P}) = {\cal L}$.
Thanks to the algorithms in the previous section, they are efficiently implementable in computer.
We compare the numbers of queries posed to the expert.

The experiment was done as follows.
The ground set $Q$ consists of $10$ elements. 
We applied the above two algorithms to $200$ instances of target antimatroids ${\cal L}$, 
and compared the numbers of queries posed to identify ${\cal L}$, where 
the target antimatroid ${\cal L}$ is given 
by ${\cal L} := {\cal A}({\cal R}_0)$ for a set ${\cal R}_0$ of randomly chosen $10$ rules.
The both algorithms examine,  in the same order, queries $(A,q)$ from smaller $|A|$.
We count the number of queries that are posed to the expert. 

The result is summarized as follows.
In average, 
the first algorithm posed $2128$ queries and the second algorithm posed $1525$ queries.
Table~\ref{tab:rate} shows the distribution of instances 
with respect to the reduction of queries.
\begin{table}
	\caption{Result}
	\begin{center}
		\begin{tabular}{|l||c|}
			\hline
			Rate $r$ of cut & Number of instances \\
			\hline
			$ 0 \leq   r \leq 10$ &  1 \\
			$10 < r \leq 15$  & 2  \\
			$15 < r  \leq 20$ & 7 \\
			$20 < r \leq 25$ & 37 \\
			$25 < r  \leq 30$ & 55 \\
			$30 < r \leq 35$  & 70 \\
			$35 < r \leq 40$ & 24 \\
			$40 < r$    & 4 \\
			\hline
		\end{tabular}\label{tab:rate}
	\end{center}
\end{table}
Here the rate~$r$ of cut of queries (from original to revised) 
is defined as
\begin{equation*}
\frac{N_1 - N_2}{N_1} \times 100,
\end{equation*}
where $N_1$ and $N_2$ denote the numbers of queries posed, respectively, by the original algorithm and 
by the revised algorithm (for the same target).
In particular, the revision achieves the average cut rate $29 \%$.
This reduction of queries was mostly caused by negative inferences.
For queries $(A,q)$ with smaller $A$, the answers tend to be `no'. 
Indeed, $99 \%$ of queries posed are `no' (i.e., posed:NO);
the average number of queries that the expert said ``no" is $2107$ for the first algorithm
and is $1505$ for the second algorithm.
Consequently, the query reduction by (strong) negative inferences is 
more powerful than that by (strong) positive inferences.
In the most successful instance, the cut rate is $42 \%$. 

We also conducted the same experiment for the reverse ordering of queries, 
where queries are examined from larger $|A|$. 
(This is a difficult situation for a human expert.)
Also in this case, our revision effectively reduces the number of queries posed.
In average, the original QUERY posed $434$ queries and the revised QUERY posed $207$ queries.
Thus the cut rate is $52 \%$ in average. 
Compared with the above ordering, 
the numbers of required queries are considerably 
small for both the original and revised.
This may be caused by our random construction of instances.
Contrary to the above ordering of queries,  
posed queries tend to be `yes';
the average number of queries 
that the expert said ``yes" is $390$ for the original
and is $163$ for the revised.
Consequently strong positive inferences contribute 
the reduction of queries effectively.

These experimental results show that the revised QUERY 
has a potential to drastically reduce 
the burden of human experts 
in the antimatroid design in KST-based educational system.

\section{Concluding remarks}\label{sec:concluding}
In this paper, we have studied the representation ${\cal R} \mapsto {\cal A}({\cal R})$ of an antimatroid 
from algorithmic and Boolean function theoretic points of view, and mentioned its potential applications to 
actual educational system designs.
There remain several algorithmic questions 
that are interesting from both theoretical and practical sides.
We end this paper with some open problems and future research issues.
\paragraph{How to recognize whether ${\cal K}({\cal R})$ is an antimatroid.}
We have mainly focused on ${\cal A}({\cal R})$ that is always an antimatroid.
As we mentioned in the introduction,
an antimatroid ${\cal A}$ always 
admits a set ${\cal R}$ of rules with ${\cal A} = {\cal K}({\cal R})$.
Thus it is natural to give a characterization of a set ${\cal R}$ of rules 
such that ${\cal K}({\cal R})$ is an antimatroid, or equivalently, 
${\cal K}({\cal R}) = {\cal A}({\cal R})$.
Related to such a characterization, 
it is quite natural to consider the following decision problem:
\begin{description}
	\item[Input:] A set ${\cal R}$ of rules.
	\item[Task:] Decide whether ${\cal K}({\cal R})$ is an antimatroid. 
\end{description}
We do not know whether this problem is in NP, 
though it is not difficult to show that it is in co-NP.
\begin{prop}
	The problem of deciding whether ${\cal K}({\cal R})$ is an antimatroid is in co-NP.
\end{prop}
\begin{proof}
	Suppose that ${\cal K}({\cal R})$ is not an antimatroid, i.e., ${\cal K}({\cal R}) \neq {\cal A}({\cal R})$.
	There is $X \in {\cal K}({\cal R}) \setminus {\cal A}({\cal R})$.
	Therefore we can check that $X$ is in ${\cal K}({\cal R}) \setminus {\cal A}({\cal R})$
	by linear time membership algorithms for ${\cal K}({\cal R})$ and for ${\cal A}({\cal R})$.
	This gives a polynomial certificate for a NO instance.
\end{proof}
We have seen in Corollary~\ref{cor:A(R)=K(R)} a sufficient condition for ${\cal K}({\cal R}) = {\cal A}({\cal R})$.
Namely, if ${\cal R}$ corresponds to
the set of circuits of an antimatroid, 
then ${\cal K}({\cal R}) = {\cal A}({\cal R})$ holds.
By Dietrich's characterization of an antimatroid by circuits~\cite{Dietrich87}
(see \cite[Theorem 3.9]{KLS1991}) 
we can  determine in polynomial time 
whether ${\cal R}$ corresponds to the set of circuits of an antimatroid.
Also if all nontrivial (prime) implicates of ${\cal A}({\cal R})$
are given (e.g., by Algorithm~\ref{algo:resolution}), 
then we check whether ${\cal K}({\cal R}) = {\cal A}({\cal R})$.
But this approach never gives a polynomial time algorithm, since the number of 
all nontrivial (prime) implicates may be exponential in input size $l({\cal R})$.
Another approach is to find the base ${\cal B}$ of ${\cal K}({\cal R})$, 
and to use an algorithm of 
Eppstein, Falmagne, and Uzun~\cite{EFU09} for checking whether ${\cal B}$ 
is the base of an antimatroid. 
However $|{\cal B}|$ may be exponential of $l({\cal R})$.
See \cite[Section 3.6]{Wild14survey} 
for computational issues on prime implicants, bases, and rules (implications).

Adaricheva and Nation~\cite{AN14} introduced a notion of 
a {\em closure system with unique criticals}  ({\em UC-systems}), where 
a convex geometry (antimatroid) is a particular example of a UC-system; 
see also~\cite{Adaricheva12}. 
They showed a polynomial time  algorithm 
to decide whether a given set of rules defines a UC-system~\cite[Proposition 45]{AN04}.
This algorithm checks a necessary condition for ${\cal K}({\cal R})$ to be an antimatroid. 

\paragraph{Toward computational learning theory for antimatroids.}

Building an antimatroid by querying an expert 
in Section~\ref{sec:application} 
should also be discussed and analyzed 
from the view point of computational learning theory, 
particularly from Angluin's framework~\cite{Angluin87} of learning Boolean functions by queries; 
see a survey~\cite{SST10}. 
The problem formulation is the following.
The task is to identify ({\em learn}) a family ${\cal L}$ of subsets 
(or a Boolean function)
by a certain (logical) expression, such as CNF.
Here we are allowed to use a certain kind of an {\em oracle} that returns information 
of target ${\cal L}$.
Typical oracles are: 
\begin{description}
\item[Membership oracle:]  The query is a subset $X$. The oracle returns ``yes" if $X \in {\cal L}$, and ``no" otherwise.
\item[Equivalence oracle:] The query is a family ${\cal L}'$. The oracle returns ``yes" if ${\cal L}' = {\cal L}$, and ``no" otherwise.
If the answer is ``no", then a subset $X \in {\cal L} \triangle {\cal L}'$ is also returned.
\end{description}
There are several results on query learning of Horn functions, 
or equivalently, union-closed families.
Angluin, Frazier, and Pitt~\cite{AFP92} gave an algorithm to learn  
a Horn function (a union-closed family ${\cal L}$)
by $O(mn)$ membership and $O(m^2n)$ equivalence queries, 
where $n$ is the number of variables (the cardinality of the ground set $Q$) 
and $m$ is the number of clauses of the Horn formula (the number of a set ${\cal R}$ of rules with ${\cal L} = {\cal K}({\cal R})$). See also a recent related work~\cite{AB11}.
Frazier and Pitt~\cite{FP93} considered the entailment oracle for a Horn function, and gave an algorithm to learn  
a Horn function 
by a polynomial number of entailment and equivalence queries,
where the entailment oracle is:
\begin{description}
\item[Entailment oracle:]  The query is a rule $(A,q)$. The oracle returns ``yes" 
if $(A,q)$ is an implicate of ${\cal L}$, and ``no" otherwise.
\end{description}
The problem of building a space of knowledge states by querying an expert, considered in Section~\ref{sec:application}, may be formulated mathematically as the problem of 
learning a Horn function by the entailment oracle.
In our setting, it is practically impossible to let a human expert play the equivalence oracle, and
it may be difficult to apply these results to actual educational system designs. 
Nevertheless it is quite interesting   
to develop
a practically feasible and theoretically efficient learning algorithm for
spaces of knowledge states, particularly antimatroids, 
from the viewpoint of computational learning theory.
It should be noted that the above learning algorithms identify target 
space ${\cal L}$ by ${\cal L} = {\cal K}({\cal R})$.
In the case where ${\cal L}$ is assumed to be an antimatroid,
it is natural to identify ${\cal L}$ by ${\cal L} = {\cal A}({\cal R})$, 
as in our revised QUERY algorithm.
In this setting, an alternative learning algorithm 
with a better theoretical guarantee may be possible.

\paragraph{Largest extension of an antimatroid.}
Adaricheva and Nation~\cite{AN04} showed (in dual form) that
for any antimatroid ${\cal A}$ on $Q$
there exists a unique maximal antimatroid $\overline{\cal A}$ 
on $Q$ containing ${\cal A}$ as a sublattice, 
where $\overline{\cal A}$ is called  the {\em largest extension} of ${\cal A}$.
This suggests a way of associating a set $\cal R$ 
of rules with the largest extension $\overline{{\cal A}({\cal R})}$ of ${\cal A}({\cal R})$. 
A natural question is:
How can we handle $\overline{{\cal A}({\cal R})}$ by ${\cal R}$ efficiently ?
A naive membership algorithm for $\overline{{\cal A}({\cal R})}$ 
obtained from the definition \cite[p.199 (E)]{AN04}
requires checking a condition for all members of ${\cal A}({\cal R})$, 
and is far from polynomial. 
An algorithmic theory 
for association ${\cal R} \mapsto \overline{{\cal A}({\cal R})}$ 
as well as its application to KST
will deserve an interesting future research.

\section*{Acknowledgments}
We thank Jean-Paul Doignon for helpful comments and references~\cite{Doignon14ICFCA,Dowling93,K1998,Wild14} of KST, and  
the referees for helpful comments and references~\cite{AN04,AN14,BM10,KN13} of implicational systems. 
We also thank Marcel Wild for remarks.
The first author was partially supported by JSPS KAKENHI Grant Numbers 25280004, 26330023, 26280004.
The second author was partially supported by JSPS KAKENHI Grant Number 26330023 and 
the JST, ERATO, Kawarabayashi Large Graph Project.
The third author was partially supported by JSPS KAKENHI Grant Numbers 24106002, 26280001.　


\appendix
\section{Appendix}
\subsection{Relation to Horn functions}
We summarize a relation to Horn functions 
in Boolean function theory; see \cite[Chapter 6]{Booleanfunctions} for detail.
A {\em Boolean function} is a $\{0,1\}$-valued function $f$ defined on $\{0,1\}^n$, i.e.,  $f: \{0,1\}^n \to \{0,1\}$.
A family ${\cal K}$ on $Q = \{1,2,\ldots,n\}$ is identified with a Boolean function $f_{\cal K}$ 
defined by $f_{\cal K}(x_1,x_2,\ldots,x_n) = 1$ if $\{i \mid x_i = 1\} \in {\cal K}$ and $0$ otherwise.

A Boolean function $f$ is called {\em Horn} if it is represented as a {\em Horn CNF}:
\[
f(x_1,x_2,\ldots,x_n) = \bigwedge_{i=1}^m \left( 
\bigvee_{j \in P_i} x_j \bigvee_{j \in N_i} \bar x_j \right),
\]
where $P_i$ and $N_i$ are disjoint subsets of $\{1,2,\ldots,n\}$ with $|P_i| \leq 1$ for $i=1,2,\ldots,m$.
Here $\bar x_i = 1 - x_i$,  $x_i \vee x_j := \max \{x_i, x_j\}$, and $x_i \wedge x_j := \min \{x_i, x_j\}$.
For $A  = \{i_1,i_2,\ldots,i_k\}$ and $B = \{j_1,j_2,\ldots,j_l\}$,  
$\bigvee_{i \in A} x_j \bigvee_{j \in B} \bar x_j$ denotes $x_{i_1} \vee x_{i_2} \vee \cdots \vee x_{i_k} 
\vee \bar x_{j_1} \vee \cdots \vee \bar x_{j_l}$.

Suppose that $Q = \{1,2,\ldots,n\}$.
For a Boolean function $f$, 
let $T(f)$ denote the set of points $x \in \{0,1\}^n = 2^{Q}$ with $f(x) = 1$.
For a rule $(A,q)$ with $A = \{p_1,p_2,\ldots,p_k\}$, the Boolean function $f_{A,q}$ is defined by
\[
f_{A,q}(x_1,x_2,\ldots,x_n) = \bar x_{p_1} \vee \bar x_{p_2} \vee \cdots \vee \bar x_{p_k} \vee x_q.
\]
For a set ${\cal R}$ of rules, we obtain a Horn function $f_{\cal R}$ by
\begin{equation*}
f_{\cal R} := \bigwedge_{(A,q) \in {\cal R}} f_{A,q}
\end{equation*} 
Then the family ${\cal K}({\cal R})$ and the Horn function $f_{\cal R}$ are equivalent objects in the following sense.
For $x \in \{0,1\}^n$, let $I^0(x)$ denote the set of elements $i \in Q$ with $x_i = 0$ ($0$-supports).
\begin{lem}
	For a set ${\cal R}$ of rules, it holds
	$
	{\cal K}({\cal R}) = \{  I^0(x) \mid x \in T(f_{\cal R})\}.
	$
\end{lem}
Notice that if we associate a rule $(A,q)$ $(A = \{p_1,p_2,\ldots,p_k\})$ 
with $x_{p_1} \vee x_{p_2} \vee \cdots \vee x_{p_k} \vee \bar x_q$, 
then a set ${\cal R}$ of rules corresponds to a {\em dual Horn function} $g_{\cal R}$, 
and ${\cal K}({\cal R})$ corresponds to the set of $1$-supports of $T(g_{\cal R})$.

In particular we can use known results of Horn functions to ${\cal K}({\cal R})$.
For example, deciding whether $(A,q)$ is an implicate for ${\cal K}({\cal R})$ 
is equivalent to deciding whether an arbitrary $x \not \in T(f_{A,q})$ satisfies $x \not \in T(f_{\cal R})$.
This is solved by the following. Fix variables $x_q = 0$ and $x_i = 1$ for $i \in A$. 
Substitute them to $f_{\cal R}$ and then obtain another Horn CNF $f'$.
If $f'$ is satisfiable (i.e., $\exists x \in T(f')$), then the answer is NO. 
Otherwise the answer is YES.
It is well-known that the satisfiability problem for Horn CNF is efficiently solved; 
see~\cite[Section 6.4.1]{Booleanfunctions}.

\subsection{Proof of Proposition~\ref{prop:AE}}
	(1).  For the set ${\cal C}$ of all circuits and the corresponding set ${\cal R}$ of rules, 
	it holds ${\cal A} = {\cal L}({\cal C}) = {\cal A}({\cal R}) \subseteq {\cal K}({\cal R})$ by Theorems~\ref{thm:rep_Horn_circuit} and \ref{thm:circuit}.
	This means that for any circuit $(C,r)$, 
	the rule $(C - r ,r )$ is actually an implicate of ${\cal A}$.
	We next show that $(C - r ,r )$ is prime.
	Let $a$ be an arbitrary element of $C - r$.
	By definition, 
	$C - a$ is free, 
	and hence there exists a member $K$ of $\mathcal{A}$ 
	such that $K\cap (C - a)=\{r\}$.
	Then $K$ contains $r$ and satisfies $K\cap ( C - a - r)=\emptyset$.
	This means that $(C - a -r, r)$ is not implicate of ${\cal A}$.
	Thus $(C - r ,r )$ is a prime implicate.

	(2). Next, suppose that $(A,q)$ is a prime implicate. 
	Namely, for any element $a$ of $A$, the rule $(A - a, q)$ is not an implicate of $\mathcal{A}$. 
	Notice that $A+q$ is not free,  
	since there is no member $K$ of ${\cal A}$ with $(A + q) \cap K = \{q\}$.
	Choose an arbitrary element $x$ of $A + q$.
	We are going show that $A + q - x$ is free; 
	then it follows that $(A + q, q)$ is a circuit by definition.
	Here  $\mathcal{A} \cap X$ denotes $\{X\cap K\mid K\in\mathcal{A}\}$ for simplicity.
	We first consider the case where $x=q$. Let $a\in A$. 
	By assumption, $(A - a,q)$ is not an implicate of $\mathcal{A}$. 
	Then there exists a member $K$ of $\mathcal{A}$ such that $q\in K$ and 
	$K \cap (A - a)=\emptyset$. But $K\cap A \neq \emptyset$, since $(A, q)$ is an implicate of $\mathcal{A}$. Therefore, $K\cap A=\{a\}$ and $\mathcal{A} \cap A$ contains $\{a\}$. 
	Since $\mathcal{A}$ is closed under union, so is $\mathcal{A} \cap A$. 
	It follows that $A + q - q=A$ is free.

	We next consider the case where $x\in A$. Let $a\in A - x$. By the same discussion as above, there exists a member $K$ of $\mathcal{A}$ such that $q\in K$ and $K\cap A=\{a\}$. We take a minimal such $K$. 
	By the minimality of $K$ and the accessibility of ${\cal A}$, 
	at least $K - a$ or $K - q$ must be a member of $\mathcal{A}$. 
	But $K - a$ is not a member of $\mathcal{A}$, 
	since $(A, q)$ is an implicate of $\mathcal{A}$ and 
	we have $q\in K - a$ and $(K - a)\cap A=\emptyset$. 
	Hence $K - q$ is a member of $\mathcal{A}$. 
	Since $(K - q)\cap (A - x + q)=K\cap (A- x)=\{a\}$, 
	the family $\mathcal{A} \cap (A - x+ q)$ contains $\{a\}$. 
	We show that $\mathcal{A} \cap (A - x + q)$ also contains $\{q\}$. Since $(A, q)$ is an implicate of $\mathcal{A}$ but $(A - x,q)$ is not, there exists a member $L$ of $\mathcal{A}$ such that $q \in L$ and $L\cap A=\{x\}$. 
	We have $L\cap (A - x + q)=\{q\}$. 
	Therefore, $\mathcal{A} \cap (A - x + q)$ contains all singletons $\{a\}$ 
	with $a\in A - x + q$. Since $\mathcal{A} \cap (A - x + q)$ is closed under union, 
	the set $A - x + q= A+q - x$ is free. 
	Thus we conclude that $(A + q, q)$ is a circuit.

\end{document}